\documentclass[reqno]{amsart}
\usepackage[height=9in,a4paper,hmargin={2.45cm,2.45cm}]{geometry}
\usepackage[T1]{fontenc}
\usepackage[utf8]{inputenc}
\usepackage[english]{babel}
\usepackage{amsfonts,amssymb,amsmath,amsthm,mathtools,mathrsfs}
\usepackage{esint}
\usepackage{enumerate,enumitem,xcolor}
\usepackage{pgfplots,float}
\usepgfplotslibrary{groupplots}

\definecolor{blue_links}{RGB}{13,0,180}
\usepackage{hyperref}
\hypersetup{
  colorlinks=true,
  linktoc=all,
  linkcolor=blue_links,
  citecolor=blue_links,
  urlcolor=blue_links,
}

\numberwithin{equation}{section}
\theoremstyle{plain}
\newtheorem{theorem}{Theorem}[section]
\newtheorem{proposition}[theorem]{Proposition}
\newtheorem{lemma}[theorem]{Lemma}
\newtheorem{corollary}[theorem]{Corollary}
\theoremstyle{definition}
\newtheorem{remark}[theorem]{Remark}

\newcommand{\R}{\mathbb R}
\newcommand{\Z}{\mathbb Z}
\newcommand{\N}{\mathbb N}
\newcommand{\T}{\mathbb T}
\newcommand{\HH}{\mathcal H}
\newcommand{\LL}{\mathcal L}
\newcommand{\A}{\mathcal A}
\newcommand{\E}{\mathcal E}
\newcommand{\trn}{\mathscr T}
\newcommand{\varn}{\mathscr A}
\newcommand{\defect}{\mathscr D}
\newcommand{\dd}{\,\mathrm d}
\newcommand{\one}{\mathbf 1}
\newcommand{\sgn}{\operatorname{sign}}
\newcommand{\sech}{\operatorname{sech}}

\newcommand{\supp}{\operatorname{supp}}
\newcommand{\diam}{\operatorname{diam}}

\newcommand{\Fp}{m_{\lambda,h}}
\newcommand{\Cosh}{C_{\lambda,h}}
\newcommand{\Alh}{A_{\lambda,h}}
\newcommand{\mean}[1]{\overline{#1}}

\title[Global minimality of lamellae for isotropic Ohta--Kawasaki and Yukawa energies]{Global minimality of lamellae for isotropic Ohta--Kawasaki and Yukawa energies}
\author[D. Reggiani]{Dario Reggiani}
\address[Dario Reggiani]{Applied Mathematics M\"unster, University of M\"unster\\
	Einsteinstrasse 62, 48149 M\"unster, Germany}
\email{dario.reggiani@uni-muenster.de}

\begin{document}

\begin{abstract}
	We establish a sharp comparison principle for the sharp-interface
	Ohta--Kawasaki energy and its Yukawa-screened counterpart in dimensions
	$2\le d\le 7$, for $m=0$, and under periodic and zero-flux Neumann boundary conditions.
	Whenever the finite optimal stripe width is compatible with the
	domain, we prove global minimality of periodic lamellae and
	characterize all minimizers. The proof combines radial
	comparison kernels, a sharp Fourier bound, and rigidity for a
	directional translation inequality in $BV$. The kernels are positive
	measures up to dimension five; in dimensions six and seven,
	distributional lifting and monotonicity of the averaged translation
	defect yield the same sharp comparison. In the Appendix, we also prove that no finite
	positive measure with the prescribed Fourier multiplier exists in
	dimensions $d\ge6$, 
	and explain why the argument does not directly extend to $d \geq 8$.
\end{abstract}

\maketitle

\subsection*{AI usage statement}
\noindent The central ideas of the proof were produced by OpenAI's ChatGPT-6 Astra. Specifically, the model proposed the construction of the auxiliary measure (Section \ref{sec:measure}), the comparison inequality (Section \ref{sec:proofs}), and the reduction of the Neumann case to the periodic one (Section \ref{sec:proofs}). The model was given a document (Ideas.pdf) by the author summarizing his previous unsuccessful approaches to the problem and why they failed, together with the preprint \cite{CG26}. It was asked to treat a simplified test case, after which the author requested a global comparison against all admissible competitors. The author verified all arguments, rewrote and completed the proofs, structured the paper, wrote the introduction, the discussion of the literature and the overview of the proof (Subsection \ref{susec:ideas}), and Appendix \ref{app:dimension}. LLM assistance was also used in drafting the text. A detailed account of the original interaction together with Ideas.pdf is available as an ancillary
file accompanying the arXiv version of this paper. The extension to $d=6,7$ arose from a later, informal exchange with the same model, prompted by the author's question whether positivity of the measure could be relaxed to a controlled amount of negativity. The model then observed that it suffices for the pairing with the geometric defect to be
non-negative. The author takes full responsibility for the content of this paper.

\section{Introduction}\label{sec:introduction}

The competition between interfacial energy and nonlocal repulsion
is a fundamental mechanism underlying pattern formation and symmetry
breaking. Interfacial energy favours the formation of large domains,
whereas repulsive interactions favour their subdivision into smaller
regions. The balance between these effects can select a characteristic
length scale and give rise to ordered patterns, including droplets,
cylinders, and lamellae.
The formation of layered patterns selects a preferred direction,
even when the interfacial energy and the interaction kernel are isotropic.
A central example is the Ohta--Kawasaki model,
introduced in \cite{OK86} to describe microphase separation in diblock
copolymer melts.

A diblock copolymer consists of two chemically distinct polymer
blocks joined by a covalent bond. The incompatibility of the two
blocks favours segregation, while their connectivity prevents
separation into macroscopically large pure phases.
The relative length of the two blocks determines the morphology.
For symmetric copolymers, lamellar phases of alternating layers are
observed experimentally and predicted by mean-field theory, whereas
asymmetric compositions lead to cylindrical, spherical or gyroid
structures.
The Ohta--Kawasaki functional describes this mechanism through a
local phase-separation energy coupled to a nonlocal term of
Coulomb type (see the derivation from a mean-field description in
\cite{CR03}).

In the standard formulation, on a domain $\Omega\subset\mathbb{R}^d$,
the diffuse-interface energy takes the form
\[
\int_\Omega
\left(
\frac{\varepsilon}{2}|\nabla u|^2
+ \frac{1}{\varepsilon}W(u)
\right)\,dx
+
\gamma
\bigl\langle u-m,(-\Delta)^{-1}(u-m)\bigr\rangle_{L^2(\Omega)}.
\]
Here $u$ is an order parameter describing the local composition,
$W$ is a suitable double-well potential, $\varepsilon$ measures the
thickness of the transition layers, and
\[
m=\frac{1}{|\Omega|}\int_\Omega u\,dx
\]
is the prescribed average composition, which reflects the relative
length of the two blocks; the symmetric case corresponds to $m=0$.
In the sharp-interface description, the order parameter takes
values in $\{-1,1\}$ and the local contribution becomes an
interfacial energy. Absorbing the surface-tension constant into
$J>0$, the corresponding functional is
\[
J|Du|(\Omega)
+
\gamma
\bigl\langle u-m,(-\Delta)^{-1}(u-m)\bigr\rangle_{L^2(\Omega)},
\qquad
u\in BV(\Omega;\{-1,1\}).
\]
A closely related class of models is obtained by replacing the
Coulomb interaction with a screened Coulomb, or Yukawa, interaction.
At the operator level, this amounts to replacing the inverse
Laplacian by
\[
R_\lambda=(-\Delta+\lambda^2)^{-1},
\qquad \lambda>0.
\]

Screened repulsive interactions are particularly relevant to
charged colloidal suspensions and models of protein solutions.
In these systems, short-range attraction can compete with
electrostatic repulsion screened by the surrounding medium,
leading to aggregation and spatially modulated structures
(see \cite{Stradner04,CCA15}).

The present paper studies the sharp-interface energy with zero
preferred composition, namely
\[
\mathcal E_{\lambda,J,\gamma}(u)
=
J|Du|(\Omega)
+
\gamma\langle u,R_\lambda u\rangle_{L^2(\Omega)},
\]
on $\Omega=\mathbb{T}^d_L$ with periodic boundary conditions, or on
$\Omega=(0,L)^d$ with homogeneous Neumann boundary conditions.
For $\lambda=0$, we work with zero-average configurations and
interpret $R_0$ as the inverse Laplacian on that subspace.
For $\lambda>0$, we retain the constant mode of the resolvent
and impose no average constraint: zero is the preferred composition,
while the actual average of $u$ is free to vary.
The constant mode then contributes with a term penalizing an imbalance
between the two phases without forbidding it
(see Corollary \ref{cor:mean}).
In the compatible case, minimizers are lamellae and therefore
automatically have zero average.
Our interest is in the exact selection of lamellar configurations:
periodically alternating parallel layers whose width is determined
by the competition between the two terms.

\subsection{Overview of the literature}
\label{subsec:scientific-context}

Early rigorous results on periodic pattern selection were obtained
in one dimension. M\"uller \cite{Mu93} proved periodicity of minimizers
for a related singularly perturbed variational problem.
For the diblock copolymer functional, Ren and Wei
\cite{RW00,RW03} studied local minimizers and the emergence of
increasingly fine periodic oscillations in singular limits.
We also refer to \cite{FL19} for the proof of periodicity of minimizers
for a one-dimensional liquid drop model.

In higher dimensions, the difficulty is of a different nature.
The isotropy of the interfacial energy and interaction kernel
supplies no preferred orientation in the ambient space.
Proving that minimizers are lamellar requires excluding arbitrary
multidimensional competitors, including configurations with curved
or rearranged layers.
A fundamental result of Alberti, Choksi, and Otto \cite{ACO09}
shows that the energy of minimizers of the sharp-interface
Ohta--Kawasaki functional is approximately uniformly distributed
over subdomains that are large compared with the intrinsic length
scale (see also \cite{Spa09}).

Concerning critical points and local minimality, Ren and Wei
\cite{RW05} studied wriggled lamellar solutions and their stability.
Choksi and Sternberg \cite{CS07} derived the first and second
variations of the sharp-interface nonlocal isoperimetric functional.
Acerbi, Fusco, and Morini \cite{AFM13} proved that strict positivity
of the second variation implies quantitative local minimality with
respect to $L^1$ perturbations (see also \cite{JP17} for the Neumann
case and \cite{Fio26} for related anisotropic local and global
minimality results for a lamella).
Cristoferi \cite{Cri18} constructed locally minimizing periodic
configurations near strictly stable periodic constant-mean-curvature
hypersurfaces, and Acerbi, Chen, and Choi \cite{ACC22} obtained further
criteria for the local minimality of lamellae.

Global minimality has been established in particular geometric or
parameter regimes. Sternberg and Topaloglu \cite{ST11} proved that, on
the flat two-dimensional torus, a single lamellar region is globally
minimizing for sufficiently weak nonlocal interaction and an
appropriate range of volume fractions. Morini and Sternberg
\cite{MS14} proved that minimizers in sufficiently thin domains are
striped, with transitions in the optimal number of layers as the
interaction strength varies. In both cases, the domain geometry or the
closeness to the local isoperimetric problem allows a complete
description of minimizers.

A complementary literature concerns the small-mass and dilute regimes.
The geometry of individual droplets and the minimality of balls at
small mass are studied in
\cite{CiSp13,KM13,KM14,Jul14,BC14,FFMMM15,CR25}.
Effective energies, droplet distributions, and confinement are instead
analyzed in \cite{M10,CP10,CP11,GMS13,GMS14,KMN16,KMN20,ABCT20,FLS26}.

An analogous competition arises in spin systems with short-range
ferromagnetic and long-range antiferromagnetic interactions, where
lamellar patterns are expected to minimize the energy in suitable
regimes near the ferromagnetic transition, that is, when $J$
approaches a critical value $J_c$. For these systems and related
mean-field models, reflection positivity and localization have led
to periodic-ordering results \cite{GLL06,GLL09}, asymptotic optimality
of stripes \cite{GLS14}, and exact striped ground states near the
transition for sufficiently rapidly decaying interactions \cite{GS16}.

In the continuum setting, the main line of results
is based on a blow-down near the critical interaction strength.
As $J\to J_c^-$, the optimal period diverges. After rescaling lengths
by this period, the energy $\Gamma$-converges to a limiting functional
whose structure encodes the competition between the local and the
nonlocal terms.
For energies with an $\ell^1$-anisotropic perimeter and power-law
interactions with $p>2d$,
Goldman and Runa \cite{GR19} identified such a limiting functional
and showed that its only finite-energy states are striped.
Daneri and Runa \cite{DR19} then transferred this rigidity back to
the original functional for $J$ close to $J_c$ and $p\ge d+2$.
For the relevant exponents the limiting functional penalizes corners
infinitely, which forces interfaces to be flat and aligned with the
preferred directions. This yields exact periodic stripe formation
(see also \cite{Ker22}).
Related results extend to diffuse-interface models \cite{DKR22}.
For Yukawa-type interactions with $\ell^1$ anisotropy \cite{DR20},
the corresponding limiting penalty acts as a hard-core constraint
between neighbouring interfaces.
For the screened energy considered here, the critical value is
$J_c=\gamma\lambda^{-3}$. If $h_*$ denotes the optimal stripe width,
the relation $J=\gamma j_\lambda(h_*)$, with $j_\lambda$ defined
in \eqref{eq:j}, gives
$\lambda h_*\to\infty$ if and only if $J/J_c\to1^-$.
Thus the analogous near-critical regime corresponds to stripe
widths much larger than the screening length. 
Finally, Daneri and Runa \cite{DR24} obtained the first rigorous
proof of spontaneous lamellar pattern formation for fully isotropic
continuum energies. For power-law kernels with exponent $p\ge d+3$,
the limiting functional penalizes the curvature of interfaces
infinitely, so its finite-energy states are lamellar.
As a consequence, for $J$ sufficiently close to $J_c$, minimizers are
close in $L^1$, after rescaling, to periodic lamellae, and a
stability argument around stripes then upgrades this
closeness to exact global minimality of periodic lamellae.
In the isotropic screened case considered here, the near-critical regime is expected to be accessible, at least in spirit, by the methods of \cite{DR24}.

Despite this progress, to the best of our knowledge, a sharp lower
bound by the optimal stripe energy density has remained unavailable
for the isotropic Ohta--Kawasaki and Yukawa energies throughout the
parameter range admitting a finite optimal stripe width.
The question addressed here is whether such a comparison can be
established without a near-critical assumption or a thin-domain
reduction, and whether its equality cases can be classified when
the optimal width is compatible with the domain.

\subsection{Overview of the results and proof strategy}

In the present paper we establish a sharp lower bound for the screened
and unscreened energies in dimensions $2\le d\le7$, under periodic and
zero-flux Neumann boundary conditions. The comparison holds for every
admissible binary $BV$ configuration and yields the global minimality
of periodic lamellae whenever their optimal width is compatible with
the domain. We also characterize all equality cases, obtaining a
complete description of the global minimizers under the corresponding
compatibility condition.

The relevant length scale is the width $h_*$ that minimizes the
energy density among periodic stripes. It is characterized by
\[
J=\gamma j_\lambda(h_*),
\]
where $j_\lambda$ is defined in \eqref{eq:j}. In the unscreened case,
this gives the explicit scale
\[
h_*=\left(\frac{12J}{\gamma}\right)^{1/3}.
\]
In the screened case, a finite optimal width exists precisely when
$0<J/\gamma<\lambda^{-3}$, and it is unique. Our main estimate is
\[
\E_{\lambda,J,\gamma}^{\mathrm b}(u)
\ge \gamma\sigma_\lambda(h_*)L^d,
\qquad \mathrm b\in\{\mathrm{per},N\},
\]
where $\gamma\sigma_\lambda(h_*)$ is the optimal stripe energy density,
given by \eqref{eq:sigma}. Thus the value obtained by optimizing the
one-dimensional stripe family provides a lower bound for the full
multidimensional problem.

On the torus $\T_L^d$, the lower bound is attained if and only if
there exists $a\in\Z^d\setminus\{0\}$ such that
$|a|=L/(2h_*)$. Whenever this condition holds, all global minimizers
are of the form
\[
u(x)=\sgn\cos\!\left(\frac{2\pi}{L}a\cdot x+\varphi\right),
\qquad
|a|=\frac{L}{2h_*},\quad \varphi\in\R.
\]
The admissible directions therefore include oblique stripes,
according to the geometry of the frequency lattice.
For the Neumann problem on $Q_L=(0,L)^d$, equality holds precisely
when $n=L/h_*$ is a positive integer, and the complete class of
global minimizers is
\[
u(x)=\pm\sgn\cos\!\left(\frac{\pi n x_j}{L}\right),
\qquad j\in\{1,\ldots,d\}.
\]
The boundary phase intervals have width $h_*/2$, while consecutive
interfaces are separated by $h_*$. In the screened case, no
constraint on the average is imposed: attainment of the sharp
bound itself forces the two phases to have equal volume.
The precise statements are given in
Theorems~\ref{thm:periodic} and~\ref{thm:neumann}.

The bound comes from an auxiliary radial measure or distribution~$\mu$,
which acts as a \emph{dual certificate}. It is constructed so that the
two competing terms of the energy can be compared without loss along
stripes. The unscreened case $\lambda=0$ is treated directly by the
same approach. See Subsection~\ref{susec:ideas} for a detailed explanation.

The comparison starts from the $BV$ translation inequality
(see \cite{Gal11,Mil26}), which bounds translation differences by
the directional variation. In dimensions $d\le5$, averaging this
inequality against a positive measure~$\mu$ bounds the averaged
translations by the perimeter, with coefficient given by the first
absolute directional moment of~$\mu$. In dimensions six and seven,
$\mu$ is instead a distribution. The comparison remains valid because
the spherical average of the translation defect is nonnegative and
nondecreasing, and the distribution has a nonnegative action on this
class of radial functions. A one-sided Fourier bound
$\widehat\mu\le\widehat G_\lambda$ provides the corresponding comparison
with the nonlocal term. Together, these yield an exact identity:
the energy excess over $\gamma\,\sigma_\lambda(h_*)L^d$ is the sum
of a geometric defect and a spectral defect, both nonnegative
(see \eqref{eq:exactdefect}).

The properties of~$\mu$ are guided by requiring both defects to vanish
on the optimal periodic stripes. The translation defect vanishes for
translations of length at most the stripe width, motivating the
support condition
$\operatorname{supp}\mu\subset\overline B_{h_*}$.
Stripes have Fourier support on the odd harmonics of their period,
so vanishing of the spectral defect requires
$\widehat\mu=\widehat G_\lambda$ at those frequencies.
We first construct the one- and two-dimensional profiles and then
obtain the higher-dimensional objects through the Fourier dimension
walk, preserving the radial Fourier multiplier
(Proposition~\ref{prop:measure}).
In this respect, the method resembles the Delsarte and
Cohn--Elkies bounds~\cite{CE03,CKMRV17,V17}: constraints in physical
and Fourier space produce a bound attained by the conjectured
optimizer. Here the certificate is compactly supported and satisfies
a one-sided Fourier bound, while positivity is required only through
its action on the geometric defect.
In the equality case, vanishing of the geometric defect forces
complete parallel interfaces separated by at least $h_*$.
The vanishing spectral defect then fixes every layer width to be
exactly $h_*$. Even reflection transfers the comparison and its
equality cases to the Neumann problem, where the reflection
symmetries determine the stripe direction and phase.

The dimensional restriction reflects the regularity of these
comparison kernels. The one-dimensional profile has a derivative
jump at the edge of its support, and each lift by two dimensions
introduces one radial derivative. Along the odd-dimensional chain,
this produces a jump in the density in dimension three and a surface
measure in dimension five. Positive finite measures with the
prescribed Fourier multiplier cannot exist in dimensions $d\ge6$,
as shown in Appendix~\ref{app:dimension}. Nevertheless, one further
lift preserves nonnegativity of the pairing with the geometric
defect, due to the nonnegativity and monotonicity of its spherical
average. This extends the comparison to $d\le7$. Further lifts
introduce second and higher radial derivatives, whose contributions
are not controlled by these properties alone. Extending the argument
to higher dimensions would likely require additional quantitative
information on the geometric defect, beyond nonnegativity and
monotonicity, possibly combined with a different construction of the
comparison measure. We do not pursue this here.

\subsection*{Acknowledgments} 
\noindent This work was supported by the Deutsche Forschungsgemeinschaft (DFG, German Research Foundation)  under the Germany's Excellence Strategy EXC 2044/2-390685587, Mathematics M\"unster: Dynamics--Geometry--Structure.

\section{Notations and preliminaries}

\noindent\emph{Geometric notation.}
The symbols $|x|$ and $x\cdot y$ denote the Euclidean norm and
scalar product, and $e_1,\ldots,e_d$ are the coordinate unit vectors.
For $r>0$ and $x\in\R^d$, we set
\[
 B_r(x):=\{y\in\R^d:|y-x|<r\},\qquad
 B_r:=B_r(0),\qquad \mathbb S^{d-1}:=\partial B_1.
\]
We write $B_r^d$ when the ambient dimension needs to be specified.
For $e\in\mathbb S^{d-1}$, the orthogonal complement of $e$ is
$e^\perp:=\{y\in\R^d:y\cdot e=0\}$.

We denote by $\LL^d$ the Lebesgue measure on $\R^d$ and by
$\HH^k$ the $k$-dimensional Hausdorff measure, normalized to agree
with $k$-dimensional Lebesgue measure on affine $k$-planes.
For a measurable set $E$, we write $|E|:=\LL^d(E)$,
$\one_E$ for its characteristic function, and
$E\triangle F:=(E\setminus F)\cup(F\setminus E)$ for symmetric
difference. Equalities between measurable sets or functions are
understood up to Lebesgue-null sets unless otherwise stated. We set $\omega_{d-1}:=\HH^{d-1}(\mathbb S^{d-1})$. Given a set $E \subset \R^d$, $\mathrm{co}E$ indicates its convex hull in $\R^d$.

For $\ell>0$, the flat torus
$\T_\ell^d:=\R^d/(\ell\Z)^d$ is endowed with the measure induced
by Lebesgue measure, so $|\T_\ell^d|=\ell^d$.
A function on $\T_\ell^d$ is identified with its $\ell$-periodic
lift to $\R^d$, and we use the same symbol for both.
Integration over the torus means integration over a fundamental
cell, for instance $[0,\ell)^d$. Translations on the torus are
understood through this periodic lift.
If $\Omega$ is a measurable domain or a torus with
$0<|\Omega|<\infty$, the average of $u\in L^1(\Omega)$ is
\[
 \mean u:=\frac{1}{|\Omega|}\int_\Omega u\dd x.
\]
The domain of the average will always be clear from the context.
The pairing $\langle u,v\rangle$ denotes $\int_\Omega uv\dd x$
for real-valued $L^2$ functions. We use the usual Sobolev space
$H^1(\Omega)$ of $L^2$ functions with first weak derivatives in
$L^2$; on a torus, functions and test functions are periodic.

\noindent\emph{Measures, restrictions, and projections.}
For a Radon measure $\mu$ and a Borel set $A$, its restriction is
$(\mu \llcorner A)(B):=\mu(A\cap B)$.
The support $\supp\mu$ of a positive measure is the set of points
whose every open neighborhood has positive $\mu$-measure.
For a Borel map $P:\R^d\to\R^k$, the pushforward $P_\#\mu$ is
defined by $P_\#\mu(A):=\mu(P^{-1}(A))$. Equivalently,
\[
 \int_{\R^k}f(y)\dd(P_\#\mu)(y)
 =\int_{\R^d}f(Px)\dd\mu(x)
\]
for every nonnegative Borel function $f$ and every positive
measure $\mu$. When $P$ is a coordinate projection, $P_\#\mu$
is also called the corresponding marginal of $\mu$.

A positive measure is \emph{radial} if $R_\#\mu=\mu$ for every
orthogonal transformation $R$, and \emph{even} if
$\mu(-A)=\mu(A)$ for every Borel set $A$.
Radial measures are even. 

\noindent\emph{Functions of bounded variation and perimeter.}
For an open set $\Omega\subset\R^d$, the space $BV(\Omega)$
consists of the functions $u\in L^1(\Omega)$ whose distributional
gradient $Du$ is a finite $\R^d$-valued Radon measure. Thus
\[
 \int_\Omega u\,\operatorname{div}\phi\dd x
 =-\int_\Omega\phi\cdot\dd Du
 \qquad\text{for every }\phi\in C_c^1(\Omega;\R^d).
\]
We denote by $|Du|$ its total variation measure and equip
$BV(\Omega)$ with the norm $\|u\|_{L^1(\Omega)}+|Du|(\Omega)$.
The notation $BV_{\mathrm{loc}}$ means that these conditions hold
locally. On a torus, the definition uses smooth periodic test
fields. In particular, the variation on a torus includes any
jump across the identified faces of a fundamental cell.
On an open cube, $|Du|(Q_L)$ is the relative variation inside
the cube and contains no boundary contribution from its faces.

On the bounded domains considered here, a measurable set $E$
has finite perimeter in $\Omega$ if
$\one_E\in BV(\Omega)$, and
$P(E;\Omega):=|D\one_E|(\Omega)$.
We write $BV(\Omega;\{-1,1\})$ for the $BV$ functions taking
only the values $-1$ and $1$ almost everywhere.
For such a function, with $E:=\{u=1\}$, one has
\[
 u=2\one_E-1,\qquad Du=2D\one_E,
 \qquad |Du|(\Omega)=2P(E;\Omega).
\]
If $\partial^*E$ denotes the reduced boundary of $E$, the
perimeter measure satisfies
$|D\one_E|=\HH^{d-1}\llcorner(\partial^*E\cap\Omega)$.

When a specific representative of a measurable set in $\R^d$
is needed, its density-one representative is
\[
 E^{(1)}:=\left\{x\in\R^d:
 \lim_{r\downarrow0}\frac{|E\cap B_r(x)|}{|B_r|}=1\right\}.
\]

\noindent\emph{Fourier conventions.}
On $T=\T_\ell^d$, with volume $V=\ell^d$, the frequency lattice is $\mathcal K_\ell:=2\pi\Z^d/\ell$.
For $U\in L^2(T)$ we use the normalized coefficients
\[
 \widehat U(q):=\frac1V\int_T U(x)e^{-iq\cdot x}\dd x,
 \qquad q\in\mathcal K_\ell.
\]
With this convention,
\[
 U=\sum_{q\in\mathcal K_\ell}\widehat U(q)e^{iq\cdot x}
 \quad\text{in }L^2(T),\qquad
 \int_T|U|^2\dd x=V\sum_{q\in\mathcal K_\ell}|\widehat U(q)|^2.
\]
In particular, $\widehat U(0)=\mean U$, so zero mean is
equivalent to the absence of the constant Fourier mode.
For real-valued $U$, one has
$\widehat U(-q)=\overline{\widehat U(q)}$.

For a finite measure $\mu$ on $\R^d$, we instead use the
un-normalized Euclidean Fourier transform
\[
 \widehat\mu(\xi):=\int_{\R^d}e^{-i\xi\cdot z}\dd\mu(z),
 \qquad \xi\in\R^d.
\]
If $d\mu=k(z)\dd z$, we also write $\widehat k=\widehat\mu$.
For an even positive measure, $\widehat\mu$ is real-valued;
for a radial measure it depends only on $|\xi|$.
If $P:\R^d\to\R^k$ is a coordinate projection, then
\[
 \widehat{P_\#\mu}(\eta)=\widehat\mu(P^{\mathsf T}\eta),
 \qquad \eta\in\R^k,
\]
where $P^{\mathsf T}\eta$ inserts zero entries in the omitted
coordinates. This identity follows directly from the
definition of pushforward.

For a finite positive measure $\mu$ and a periodic function
$U\in L^2(T)$, convolution is defined on the periodic lift by
\[
 (\mu*U)(x):=\int_{\R^d}U(x-z)\dd\mu(z).
\]
It is well-defined almost everywhere and satisfies
\[
 \|\mu*U\|_{L^2(T)}\le\mu(\R^d)\|U\|_{L^2(T)},
 \qquad
 \widehat{\mu*U}(q)=\widehat\mu(q)\widehat U(q).
\]
If $\mu$ is even, $U(x-z)$ may equivalently be replaced by
$U(x+z)$. For real-valued $U$ and even positive $\mu$, Parseval's
identity therefore gives
\[
 \langle U,\mu*U\rangle
 =V\sum_{q\in\mathcal K_\ell}\widehat\mu(q)|\widehat U(q)|^2.
\]

\noindent\emph{One-dimensional functions and striped sets.}
We call a binary function on $\R^d$ \emph{one-dimensional} or \emph{striped} if
there exist $\nu\in\mathbb S^{d-1}$ and
$g\in BV_{\mathrm{loc}}(\R;\{-1,1\})$ such that
\[
 U(x)=g(x\cdot\nu)\qquad\text{for almost every }x\in\R^d.
\]
For a function on a torus this definition refers to its periodic
lift. The positive phase is then
$\{x:x\cdot\nu\in I\}$, where $I:=\{t:g(t)=1\}$.
If $g$ is nonconstant, its jumps describe complete parallel
interfaces, and the regions between consecutive interfaces are
the layers or stripes. The distance between consecutive
interfaces is the layer width. Constant functions are included as degenerate
one-dimensional functions.

By \emph{periodic stripes of width $h>0$} we mean the profiles
\[
 U_{h,\nu,\varphi}(x)
 :=\sgn\cos\!\left(\frac{\pi}{h}x\cdot\nu+\varphi\right),
 \qquad \nu\in\mathbb S^{d-1},\quad\varphi\in\R,
\]
or their positive-phase sets. Their interfaces are
\[
 x\cdot\nu=h\left(k+\frac12-\frac{\varphi}{\pi}\right),
 \qquad k\in\Z.
\]
Both phases have layer width $h$, whereas the full period in
the normal direction is $2h$: translation by $h\nu$ exchanges
the phases, and translation by $2h\nu$ preserves them.
Such a profile defines a function on $\T_\ell^d$ precisely when
\[
 \ell\nu_i\in2h\Z\quad(i=1,\ldots,d),
 \qquad\text{equivalently}\qquad
 \frac{\ell\nu}{2h}\in\Z^d.
\]
We will use the word "width" always
referring to the distance between consecutive interfaces in the
full profile.

\section{Main results}\label{sec:results}

\subsection{The functionals and the admissible classes}
Let $Q_L:=(0,L)^d$ and $\T_L^d:=\R^d/(L\Z)^d$, with $L>0$.  

For $\lambda\ge0$ and either $\Omega=Q_L$ or $\Omega=\T_L^d$, we define the admissible classes
\begin{equation}\label{eq:admissible}
 \A_\lambda(\Omega)=
 \begin{cases}
 BV(\Omega;\{-1,1\}),&\lambda>0,\\
 \{u\in BV(\Omega;\{-1,1\}):\int_\Omega u=0\},&\lambda=0.
 \end{cases}
\end{equation}
For $\lambda>0$, let $R_\lambda^{\mathrm{per}}$ and $R_\lambda^N$ denote the inverses of the operator $(-\Delta+\lambda^2)$ with periodic and zero-flux Neumann boundary conditions, respectively. We remark that in the Neumann case, $w=R_\lambda^Nu$ means
\begin{equation}\label{eq:weakN}
 \int_{Q_L}(\nabla w\cdot\nabla\varphi+\lambda^2w\varphi)\dd x
 =\int_{Q_L}u\varphi\dd x
 \qquad\text{for every }\varphi\in H^1(Q_L).
\end{equation}
The definition for $R_\lambda^{\mathrm{per}}$ uses periodic $H^1$ functions and test functions instead.

When $\lambda=0$, the inverse is defined only if the average of $u$ on $\Omega$ is zero, and the solution is normalized imposing $\int_\Omega w=0$. On the other hand, if $\lambda>0$, then we get $\mean w=\mean u/\lambda^2$, and no additional normalization is imposed.

For $\mathrm b\in\{\mathrm{per},N\}$ then the screened and unscreened Ohta-Kawasaki energies are
\begin{equation}\label{eq:energy}
 N_\lambda^{\mathrm b}(u)=\langle u,R_\lambda^{\mathrm b}u\rangle,
 \qquad
 \E_{\lambda,J,\gamma}^{\mathrm b}(u)
 =J|Du|(\Omega)+\gamma N_\lambda^{\mathrm b}(u),
 \qquad J,\gamma>0.
\end{equation}
In particular, it holds that
\begin{equation}\label{eq:field}
 N_\lambda^{\mathrm b}(u)
 =\int_\Omega\bigl(|\nabla w|^2+\lambda^2w^2\bigr)\dd x.
\end{equation}
For $h>0$ we define the following quantities (they will be motivated by the computations in Section \ref{sec:1D})
\begin{align}
 j_\lambda(h)
 &:=
 \begin{cases}
 \displaystyle\frac{\sinh(\lambda h)-\lambda h}
 {\lambda^3(1+\cosh(\lambda h))},&\lambda>0,\\[1.5ex]
 \displaystyle\frac{h^3}{12},&\lambda=0,
 \end{cases}
 \label{eq:j} \\
 \sigma_\lambda(h)
 &:=
 \begin{cases}
 \displaystyle\frac{\cosh(\lambda h)-1}
 {\lambda^2(1+\cosh(\lambda h))},&\lambda>0,\\[1.5ex]
 \displaystyle\frac{h^2}{4},&\lambda=0.
 \end{cases}
 \label{eq:sigma}
\end{align}
Assuming $\gamma=1$, the quantity $j_\lambda(h)$ represents the parameter $J$ in \eqref{eq:energy} such that periodic stripes of width $h$ are optimal among periodic striped competitors and $\sigma_\lambda(h)$ is the energy density associated to such optimal periodic stripes.

\subsection{Periodic and Neumann minimizers}
We now present the main results of this paper, both in the screened and unscreened case as well as in the periodic and Neumann case.

\begin{theorem}[Periodic comparison and equality cases]\label{thm:periodic}
Let $2\le d\le 7$, $\lambda\ge0$, and $L,h,\gamma>0$. Set $J=\gamma j_\lambda(h)$. Then every $u\in\A_\lambda(\T_L^d)$ satisfies
\begin{equation}\label{eq:mainper}
 \E_{\lambda,J,\gamma}^{\mathrm{per}}(u)
 \ge\gamma\sigma_\lambda(h)L^d.
\end{equation}
Equality holds if and only if
\begin{equation}\label{eq:periodicminimizers}
 u(x)=\sgn\cos\!\left(\frac{2\pi}{L}a\cdot x+\varphi\right),
 \qquad
 a\in\Z^d\setminus\{0\},\quad |a|=\frac{L}{2h},\quad \varphi\in\R.
\end{equation}
Consequently, if such an integer vector exists, the minimum equals $\gamma\sigma_\lambda(h)L^d$, and \eqref{eq:periodicminimizers} is the complete class of global minimizers.
\end{theorem}

\begin{theorem}[Neumann comparison and equality cases]\label{thm:neumann}
Let $2\le d\le 7$, $\lambda\ge0$, and $L,h,\gamma>0$. Set $J=\gamma j_\lambda(h)$. Then every $u\in\A_\lambda(Q_L)$ satisfies
\begin{equation}\label{eq:mainN}
 \E_{\lambda,J,\gamma}^N(u)
 \ge\gamma\sigma_\lambda(h)L^d.
\end{equation}
Equality is possible if and only if $n=L/h$ is a positive integer. In this case its complete class is
\begin{equation}\label{eq:neumannminimizers}
 u(x)=\pm\sgn\cos\!\left(\frac{\pi n x_j}{L}\right),
 \qquad j\in\{1,\ldots,d\}.
\end{equation}
These functions are exactly the global minimizers.
\end{theorem}

\begin{remark}\label{rem:scope}
For $\lambda>0$, neither theorem assumes zero mean. Attainment of the sharp comparison bound forces it (see Corollary \ref{cor:mean}). On the other hand, for $\lambda=0$, zero mean is necessary for the inverse Laplacian to exist.
\end{remark}

\begin{remark}\label{rem:L}
The condition $L/(2h)\in\N$ is sufficient for coordinate-aligned periodic stripes, however, being the problem isotropic, on a multidimensional torus the lattice condition in \eqref{eq:periodicminimizers} also allows oblique stripes.
\end{remark}

\subsection{Ideas of the proof}\label{susec:ideas}
In this subsection we present the main ideas behind inequalities
\eqref{eq:mainper} and \eqref{eq:mainN}.

Consider first $U \in BV(T;\{-1,1\})$ on a torus $T$ of volume
$V$. Write $DU=\theta|DU|$ and introduce
\begin{equation}\label{eq:translationdefs}
	\trn_U(z):=\int_T|U(x+z)-U(x)|\dd x,\qquad
	\varn_U(z):=\int_T|z\cdot\theta(x)|\,\dd|DU|(x).
\end{equation}
The directional translation inequality (see Lemma~\ref{lem:translation}) gives
\begin{equation}\label{eq:translationheuristic}
	\trn_U(z)\le\varn_U(z).
\end{equation}
Heuristically, the right-hand side counts every interface crossing along
a segment, while the left-hand side only records whether the two
endpoints lie in different phases. Two crossings may return to the
original phase, and this creates a strict loss in
\eqref{eq:translationheuristic}.

Suppose now that $\mu$ is an even, positive, finite measure on $\R^d$ such
that
\begin{equation}\label{eq:designmoments}
	\mu(\R^d)=\sigma_\lambda(h),\qquad
	\int_{\R^d}|z\cdot\nu|\dd\mu(z)=j_\lambda(h)
	\quad\text{for every }\nu\in \mathbb S^{d-1}.
\end{equation}
The identity $ab=1-|a-b|$ for $a,b\in\{-1,1\}$ gives
$\int_T U(x)U(x+z)\dd x=V-\trn_U(z)$. Integrating against $\mu$, and then
integrating \eqref{eq:translationheuristic} against $\mu$ and using
\eqref{eq:designmoments}, we obtain
\begin{equation}
	\langle U,\mu*U\rangle
	=\sigma_\lambda(h)V-\int_{\R^d}\trn_U(z)\dd\mu(z)
	\ge \sigma_\lambda(h)V-j_\lambda(h)|DU|(T),
	\label{eq:heuristicgeometric}
\end{equation}
where the convolution is evaluated on the periodic extension of $U$.

For a one-dimensional striped set with layer width at least $h$, we have
$\trn_U(z)=\varn_U(z)$ for every $|z|\le h$. Assume that the support of
$\mu$ is $\overline{B_h}$ and that its Lebesgue density is positive on
$B_h$. Equality in \eqref{eq:heuristicgeometric} then forces
$\trn_U(z)=\varn_U(z)$ for a.e.\ $|z|<h$, hence for all $|z|<h$ by
continuity of both sides in $z$. By the rigidity statement
(Lemma~\ref{lem:rigidity}), this happens if and only if $U$ is a striped set with layer width at
least $h$.

To control the left-hand side of \eqref{eq:heuristicgeometric} by the
energy, we need $N_\lambda^{\mathrm{per}}(U)\ge\langle U,\mu*U\rangle$. A
sufficient condition is
\begin{equation}\label{eq:designfourier}
	\widehat\mu(\xi)\le\frac1{\lambda^2+|\xi|^2}.
\end{equation}
For $\lambda=0$ this is imposed only for $\xi\ne0$, and the zero-mean
constraint removes the constant mode.

Combining \eqref{eq:designfourier} with \eqref{eq:heuristicgeometric}, we
get
\[
J|DU|(T)+\gamma N_\lambda^{\mathrm{per}}(U)
\ge \gamma\sigma_\lambda(h)V+\bigl(J-\gamma j_\lambda(h)\bigr)|DU|(T).
\]
We choose $h=h_*$ so that $\gamma j_\lambda(h_*)=J$. The perimeter terms
then cancel, and the energy is bounded below by $\gamma\sigma_\lambda(h_*)V$.
More precisely, the energy excess over $\gamma\sigma_\lambda(h_*)V$ equals
the sum of a geometric defect, $\gamma\int(\varn_U-\trn_U)\dd\mu$, and a
spectral defect, $\gamma\bigl(N_\lambda^{\mathrm{per}}(U)
-\langle U,\mu*U\rangle\bigr)$. Both are nonnegative (see
\eqref{eq:exactdefect}).

For the periodic striped set with width $h$ and normal $\nu$, the nonzero
Fourier frequencies are the odd multiples of $(\pi/h)\nu$. The spectral
defect vanishes on this configuration only if
\begin{equation}\label{eq:contact}
	\widehat\mu((2k+1)\pi\nu/h)
	=\frac1{\lambda^2+((2k+1)\pi/h)^2},
	\qquad k\in\Z.
\end{equation}
The measure $\mu$ is thus designed to give equality for the candidate
periodic striped set and a lower bound for every other competitor.

We now sketch how to construct $\mu$, starting from the one-dimensional
case. For $\lambda>0$ the whole-line fundamental solution is
\[
G_\lambda(t)=\frac{e^{-\lambda|t|}}{2\lambda},
\qquad (-\partial_t^2+\lambda^2)G_\lambda=\delta_0.
\]
We look for a symmetric finite difference
\begin{equation}\label{eq:ansatz}
	k_1(t)=aG_\lambda(t)
	-\frac b2\bigl(G_\lambda(t-h)+G_\lambda(t+h)\bigr).
\end{equation}
For $t\ge h$, we have $k_1(t)=\bigl(a-b\cosh(\lambda h)\bigr)G_\lambda(t)$,
and by evenness the same holds for $t\le-h$. Hence compact support in
$[-h,h]$ requires
\[
a=b\cosh(\lambda h).
\]
At an odd stripe frequency $\xi=(2k+1)\pi/h$ we have $\cos(\xi h)=-1$, so
\[
\widehat{k}_1(\xi)=\frac{a-b \cos(\xi h)}{\xi^2+\lambda^2}
=\frac{a+b}{\xi^2+\lambda^2}.
\]
Matching $\widehat{G}_\lambda(\xi)$ therefore requires $a+b=1$. Writing
$\Cosh:=\cosh(\lambda h)$, these two conditions give
\begin{equation}\label{eq:ab}
	a=\frac{\Cosh}{1+\Cosh},\qquad b=\frac1{1+\Cosh}.
\end{equation}
For $0\le t<h$, substituting these values gives
\[
k_1(t)=
\frac{\Cosh e^{-\lambda t}-e^{-\lambda h}\cosh(\lambda t)}
{2\lambda(1+\Cosh)}
=\frac{\sinh(\lambda(h-t))}{2\lambda(1+\Cosh)}.
\]
Therefore
\begin{equation}\label{eq:k1heuristic}
	k_1(t)=
	\frac{\sinh(\lambda(h-|t|))}
	{2\lambda(1+\Cosh)}\,\one_{\{|t|<h\}},
\end{equation}
which is positive on $(-h,h)$, and
\begin{equation}\label{eq:F}
	\Fp(\xi):=\widehat k_1(\xi)
	=\frac{\Cosh-\cos(h|\xi|)}
	{(1+\Cosh)(\lambda^2+|\xi|^2)}.
\end{equation}
Since $\Cosh-\cos(h|\xi|)\le 1+\Cosh$, with equality if and only if
$h|\xi|\in(2\Z+1)\pi$, the multiplier $\Fp$ satisfies
\eqref{eq:designfourier}, with equality exactly at the odd stripe
frequencies \eqref{eq:contact}. In particular, the inequality is strict
at $\xi=0$. A direct computation gives
\[
\int_\R k_1\dd t=\frac{\Cosh-1}{\lambda^2(1+\Cosh)},\qquad
\int_\R |t|\,k_1(t)\dd t=\frac{\sinh(\lambda h)-\lambda h}{\lambda^3(1+\Cosh)},
\]
and these are the quantities $\sigma_\lambda(h)$ and $j_\lambda(h)$ in
\eqref{eq:designmoments}.
In dimension one, the measure $\dd\mu:=k_1(t)\dd t$ therefore has all
the required properties.

The key observation is that a striped set only sees the one-dimensional
marginal of a kernel. Suppose a radial measure $\mu$ on $\R^d$ projects
onto every line to $k_1(t)\dd t$. Then its mass and its directional first
moment are those of $k_1$. By the projection-slice theorem,
$\widehat\mu(\xi)=\Fp(|\xi|)$. Hence \eqref{eq:designmoments},
\eqref{eq:designfourier} and \eqref{eq:contact} are inherited
automatically, and the only nontrivial requirement in higher dimensions
is \emph{positivity} (or, for $d=6,7$, the weaker substitute below).

Projection through two transverse dimensions links radial densities in
dimensions $d$ and $d+2$:
\begin{equation}\label{eq:codim2}
	k_d(r)=2\pi\int_r^h s\,k_{d+2}(s)\dd s,
	\qquad
	k_{d+2}(r)=-\frac{k_d'(r)}{2\pi r}\quad \text{for } 0<r<h.
\end{equation}
In particular, \eqref{eq:k1heuristic} gives the three-dimensional density
\begin{equation}\label{eq:k3heuristic}
	k_3(r)=\frac{\cosh(\lambda(h-r))}
	{4\pi(1+\Cosh)r}\,\one_{\{0<r<h\}}.
\end{equation}
Unlike $k_1$, which vanishes continuously at $|t|=h$,
the density $k_3$ has a jump at $r=h$.
Each lift by two dimensions introduces one radial derivative.
In dimension five, the distributional identity
$k_5=-k_3'/(2\pi r)$ therefore produces a positive surface
measure on $\partial B_h$ accounting for this jump.
Positive finite measures with the prescribed multiplier
cannot exist in dimensions $d\ge6$, as proved in
Appendix~\ref{app:dimension}.
This describes the main mechanism of the comparison
in dimensions $d\le5$.

For $d=6,7$, we continue the Fourier walk in the sense
of distributions and exploit additional properties
of the geometric defect.
Its spherical average
\[
\psi_U(r):=
\fint_{\mathbb S^{d-1}}
\bigl(\varn_U(re)-\trn_U(re)\bigr)\,d\HH^{d-1}(e)
\]
is nonnegative, nondecreasing and Lipschitz on
$[0,\infty)$, and belongs to $C^1((0,\infty))$
(Lemma~\ref{lem:translation}).
The weak form of the dimension walk gives
\[
	\big\langle\mu_{\lambda,h}^{(d)},
	\psi_U(|\cdot|)\big\rangle
	=
	\int_{\R^{d-2}}
	\left[
	\psi_U(|z|)
	+\frac{|z|}{d-2}\psi_U'(|z|)
	\right]\,d\mu_{\lambda,h}^{(d-2)}(z)
	\ge0.
\]
Thus, positivity of the lower-dimensional measure,
together with monotonicity of the averaged defect,
ensures that the geometric contribution remains
nonnegative in dimensions six and seven.
The Fourier multiplier, mass and radial first-moment
identity are preserved by the lift, so the same sharp
comparison follows.

Moreover, a vanishing geometric contribution forces
$\psi_U(r)=0$ for $0<r<h$, because the lower-dimensional
measure has strictly positive absolutely continuous density.
Non-negativity and continuity of the original defect
then imply $\varn_U(z)=\trn_U(z)$ for every $|z|<h$.
Consequently, the same rigidity argument applies.
The construction and the extended pairings are
established in Proposition~\ref{prop:measure}.

Finally, for $\lambda=0$ the same construction applies directly. The corresponding
kernel and multiplier are
\[
k_1(t)=\tfrac14\,(h-|t|)_+,\qquad
\Fp(\xi)=\frac{1-\cos(h|\xi|)}{2|\xi|^2},
\]
with $\sigma_0(h)=h^2/4$ and $j_0(h)=h^3/12$, so that
$h_*=(12J/\gamma)^{1/3}$.

\section{Continuous optimization of one-dimensional periodic stripes}\label{sec:1D}
We now focus on the optimization in one dimension and the derivation of the relevant quantities in \eqref{eq:j} and \eqref{eq:sigma}. Although the results in this section are already known, we present them for completeness. 

\subsection{The periodic cell problem}
For each $h>0$, let $s_h(x):=\sgn\cos(\pi x/h)$ on the one-dimensional torus $\T_{2h}^1$. Define $w_{\lambda,h}:=R_\lambda^{\mathrm{per}}s_h$ as its periodic potential, normalized to have zero mean when $\lambda=0$. We define the \emph{periodic stripe energy density} to be the energy of one full period divided by its length:
\begin{equation}\label{eq:edensitydefinition}
	\begin{split}
		e_{\lambda,J,\gamma}(h)
		:=\frac1{2h}\E_{\lambda,J,\gamma}^{\mathrm{per}}(s_h)=\frac{J}{2h}|Ds_h|(\T_{2h}^1)
		+\frac{\gamma}{2h}\int_{-h}^{h}s_h(x)w_{\lambda,h}(x)\dd x.
	\end{split}
\end{equation}
For $\lambda>0$, $w$ is even and anti-periodic, that is
\[
 w(-x)=w(x),\qquad w(x+h)=-w(x).
\]
On the positive layer $[-h/2,h/2]$, the equation and boundary values are
\[
 -w''+\lambda^2w=1,\qquad w'(0)=0,\qquad w(\pm h/2)=0.
\]
Solving the boundary-value problem gives the solution
\begin{equation}\label{eq:cellpotential}
 p_{\lambda,h}(x)=
 \frac1{\lambda^2}
 \left(1-\frac{\cosh(\lambda x)}{\cosh(\lambda h/2)}\right),
 \qquad |x|\le h/2.
\end{equation}
We extend it by setting $w(x+kh)=(-1)^kp_{\lambda,h}(x)$. Notice that the interface values are zero, and the derivatives match because
\[
 p_{\lambda,h}'(h/2)
 =-\frac{\tanh(\lambda h/2)}{\lambda}
 =-p_{\lambda,h}'(-h/2).
\]
Thus, the extension is $C^1$ and solves the equation distributionally. Moreover, anti-periodicity gives zero mean.

Using the fact that $s_hw$ has period $h$, the nonlocal energy per unit length is given by
\begin{align}
 n_\lambda(h)
 :=\frac1h\int_{-h/2}^{h/2}p_{\lambda,h}(x)\dd x
 =\frac1{\lambda^2}
 -\frac{1}{\lambda^2h\cosh(\lambda h/2)}
       \frac{2\sinh(\lambda h/2)}{\lambda}
 =\frac1{\lambda^2}
 -\frac{2}{\lambda^3h}\tanh(\lambda h/2).
 \label{eq:nlambda}
\end{align}
A similar computation for $\lambda=0$ gives
\[
 p_{0,h}(x)=\frac{h^2}{8}-\frac{x^2}{2}.
\]
The same antiperiodic extension is $C^1$, has zero mean, and solves $-w''=s_h$. Direct integration gives
\begin{equation}\label{eq:n0}
 n_0(h)=\frac1h\int_{-h/2}^{h/2}
 \left(\frac{h^2}{8}-\frac{x^2}{2}\right)\dd x
 =\frac1h\left(\frac{h^3}{8}-\frac{h^3}{24}\right)
 =\frac{h^2}{12}.
\end{equation}
There are two jumps of size two in each cell of length $2h$. The total energy density is therefore
\begin{equation}\label{eq:stripeenergy}
 e_{\lambda,J,\gamma}(h)=\frac{2J}{h}+\gamma n_\lambda(h).
\end{equation}

\subsection{Continuous minimization}
For $\lambda=0$,
\[
 e_{0,J,\gamma}'(h)=-\frac{2J}{h^2}+\frac{\gamma h}{6},
 \qquad
 e_{0,J,\gamma}''(h)=\frac{4J}{h^3}+\frac{\gamma}{6}>0.
\]
The energy diverges at both $0$ and $+\infty$. Hence, its unique global minimizer is given by
\begin{equation}\label{eq:optimal0}
 h_*=\left(\frac{12J}{\gamma}\right)^{1/3},
 \qquad
 e_{0,J,\gamma}(h_*)=\frac{\gamma h_*^2}{4}.
\end{equation}
For $\lambda>0$, define for brevity $a:=\lambda h/2$. Differentiating \eqref{eq:nlambda} yields
\[
 n_\lambda'(h)
 =\frac{2}{\lambda^3h^2}\tanh a
 -\frac1{\lambda^2h}\sech^2a.
\]
Consequently,
\begin{equation}\label{eq:ederivative}
 e_{\lambda,J,\gamma}'(h)
 =\frac{2\gamma}{h^2}
 \left[
 \frac{\tanh a-a\sech^2a}{\lambda^3}
 -\frac J\gamma
 \right]
 =\frac{2\gamma}{h^2}\left[j_\lambda(h)-\frac J\gamma\right].
\end{equation}
Above we used the expression for $j_\lambda(h)$ in \eqref{eq:j}, since
$\sinh(2a)/(1+\cosh(2a))=\tanh a$ and
$2a/(1+\cosh(2a))=a\sech^2a$. Moreover,
\begin{equation}\label{eq:jprime}
 j_\lambda'(h)=\frac{h}{2\lambda}\sech^2a\,\tanh a>0,
 \qquad
 \lim_{h \to 0^+}j_\lambda(h)=0,\quad \lim_{h \to +\infty}j_\lambda(h)=\lambda^{-3}.
\end{equation}
Thus a finite continuous optimal width exists exactly when
$0<J/\gamma<\lambda^{-3}$. It is unique and is characterized by
\begin{equation}\label{eq:tuning}
 J=\gamma j_\lambda(h_*).
\end{equation}
Indeed, the derivative in \eqref{eq:ederivative} is negative before $h_*$ and positive after it, proving global minimality within the continuous periodic stripe family. At this width $h_*$ we have
\begin{align}
 e_{\lambda,J,\gamma}(h_*)
 =\frac{\gamma}{\lambda^2}
 -\frac{2\gamma}{\lambda^3h_*}
       \left(\frac{\lambda h_*}{2}\right)
       \sech^2(\lambda h_*/2)
 =\frac{\gamma}{\lambda^2}\tanh^2(\lambda h_*/2)
 =\gamma\sigma_\lambda(h_*).
 \label{eq:optimalvalue}
\end{align}
If $J/\gamma\ge\lambda^{-3}$, the energy is strictly decreasing at every finite width and converges to $\gamma/\lambda^2$ as $h\to\infty$; and there is no minimizer in the stripe family.

\subsection{Neumann cells and compatibility with the domain}
On $(0,L)$, let $n=L/h\in\N$ and set
\[
 u_n(x):=\sgn\cos(\pi n x/L).
\]
The interfaces are at $(k+\tfrac12)h$, $k=0,\ldots,n-1$. Notice also that the endpoint phase intervals have length $h/2$, whereas all interior intervals have length $h$. There are exactly $n$ interior jumps, giving $|Du_n|((0,L))=2n$.

Restrict the potential constructed above to $(0,L)$.
By evenness and anti-periodicity of $w$, we have
\[
w'(0)=0,
\qquad
w'(L)=w'(nh)=(-1)^n w'(0)=0.
\]
Thus the restriction of $w$ to $(0,L)$ solves
\[
-w''+\lambda^2w=u_n
\quad\text{in }(0,L),
\qquad
w'(0)=w'(L)=0.
\]
Moreover, evenness and anti-periodicity imply
$w(h-x)=-w(x)$, and hence
\[
\int_0^h w(x)\,dx=0.
\]
Consequently,
\[
\int_0^L w(x)\,dx
=\sum_{k=0}^{n-1}\int_{kh}^{(k+1)h}w(x)\,dx
=\sum_{k=0}^{n-1}(-1)^k\int_0^h w(x)\,dx
=0.
\]
In particular, when $\lambda=0$, this restriction satisfies the
required mean-zero normalization.

Finally, both $s_h$ and $w$ are anti-periodic with period $h$, so their product
is $h$-periodic. Since $u_n=s_h|_{(0,L)}$ and $L=nh$, we obtain
\[
\frac1L\int_0^L u_n(x)w(x)\,dx
=\frac1h\int_0^h s_h(x)w(x)\,dx
=\frac1h\int_{-h/2}^{h/2}p_{\lambda,h}(x)\,dx
=n_\lambda(h).
\]
Hence,
\begin{equation}\label{eq:Neumann1D}
 \frac{\E_{\lambda,J,\gamma}^N(u_n)}L
 =\frac{2J}{h}+\gamma n_\lambda(h).
\end{equation}
Thus, for every $h$ with $L/h\in\N$, the Neumann stripe $u_{L/h}$ has
the same energy per unit length as the periodic stripe of width $h$.
In particular, if $L/h_*\in\N$, then $u_{L/h_*}$ attains the optimal periodic
stripe energy density $e_{\lambda,J,\gamma}(h_*)$.

\section{Construction of the comparison measure}\label{sec:measure}
In this section we construct the radial comparison measures and distributions used to derive the sharp comparison inequality. We begin with the one- and two-dimensional profiles and obtain the higher-dimensional objects through the Fourier dimension walk. We then establish their Fourier transforms and moment identities, positivity as measures up to dimension five, and non-negativity of their action on nonnegative, nondecreasing radial test functions in dimensions six and seven.

\begin{proposition}[Radial comparison measures and distributions]\label{prop:measure}
	Let $h>0$, $\lambda\ge0$, and $1\le d\le7$. There is a unique compactly supported distribution
	$\mu_{\lambda,h}^{(d)}$ on $\R^d$ with Fourier transform
	\begin{equation}\label{eq:measureFT}
		\widehat{\mu_{\lambda,h}^{(d)}}(\xi)=
		\frac{\cosh(\lambda h)-\cos(h|\xi|)}
		{(1+\cosh(\lambda h))(\lambda^2+|\xi|^2)}.
	\end{equation}
	The distribution is invariant under orthogonal transformations and supported in
	$\overline B_h$. At $\lambda=0$ the right-hand side is understood by continuity
	at $\xi=0$, where its value is $h^2/4$.
	
	For $1\le d\le5$, $\mu_{\lambda,h}^{(d)}$ is a finite positive measure whose
	absolutely continuous part has strictly positive density almost everywhere in $B_h$.
	Its mass and first absolute directional moments satisfy
	\begin{equation}\label{eq:measuremoments}
		\mu_{\lambda,h}^{(d)}(\R^d)=\sigma_\lambda(h),\qquad
		\int_{\R^d}|z\cdot\nu|\,d\mu_{\lambda,h}^{(d)}(z)=j_\lambda(h)
		\quad(\nu\in\mathbb S^{d-1}).
	\end{equation}
	
	For $d\in\{6,7\}$, $\mu_{\lambda,h}^{(d)}$ has distributional order at most one.
	Its action on smooth radial tests extends to radial functions
	$\psi(z)=\psi_0(|z|)$ with
	\[
	\psi_0\in\operatorname{Lip}([0,h+\delta])\cap C^1((0,h+\delta))
	\quad\text{for some } \delta>0,
	\]
	by the pairing formula
	\begin{equation}\label{eq:liftedradialpairing}
		\big\langle\mu_{\lambda,h}^{(d)},\psi\big\rangle
		:=\int_{\R^{d-2}}
		\left[\psi_0(|z|)+\frac{|z|}{d-2}\psi_0'(|z|)\right]
		\,d\mu_{\lambda,h}^{(d-2)}(z).
	\end{equation}
	The bracket at $z=0$ is assigned the value $\psi_0(0)$.
	In particular, this pairing is nonnegative whenever $\psi_0$ is nonnegative
	and nondecreasing on $[0,h]$.
	The mass and generalized first absolute directional moments in dimensions
	$d\in\{6,7\}$ satisfy
	\begin{equation*}
		\big\langle\mu_{\lambda,h}^{(d)},1\big\rangle=\sigma_\lambda(h),\qquad
		c_d\big\langle\mu_{\lambda,h}^{(d)},|\cdot|\big\rangle=j_\lambda(h),
		\qquad
		c_d:=\frac{\Gamma(d/2)}{\sqrt\pi\,\Gamma((d+1)/2)}.
	\end{equation*}
	Above, the first pairing uses a smooth cut-off $\chi\in C_c^\infty(\R^d)$ equal to one on a neighborhood of $\overline B_h$, while the radial pairing with $|\cdot|$ is defined by \eqref{eq:liftedradialpairing}.
\end{proposition}

\begin{proof}
	Set $\Cosh:=\cosh(\lambda h)$ and
	$\Alh:=[8\pi^2(1+\Cosh)]^{-1}$.
	We first construct the one- and two-dimensional measures and then
	obtain the higher-dimensional distributions by the two-step Fourier
	dimension walk.
		
		We shall use the following observation. Let $\mu_d$ be a compactly
		supported radial distribution on $\R^d$, and let
		$\pi_d:\R^d\to\R$ be the coordinate projection
		$\pi_d(z)=z_1$. If $(\pi_d)_\#\mu_d=\mu_1$, then
		\[
		\widehat{\mu_d}(\eta,0,\ldots,0)
		=
		\big\langle\mu_d,e^{-i\eta\pi_d(\cdot)}\big\rangle
		=
		\big\langle\mu_1,e^{-i\eta\,\cdot}\big\rangle
		=
		\widehat{\mu_1}(\eta),
		\qquad \eta\in\R.
		\]
		Here, compact support allows the distributions to act on smooth
		functions without compact support, by inserting a cutoff equal
		to one near their support.
		Moreover, orthogonal invariance of $\mu_d$ implies orthogonal
		invariance of its Fourier transform. Thus, for every $\xi\in\R^d$,
		\[
		\widehat{\mu_d}(\xi)
		=
		\widehat{\mu_d}(|\xi|,0,\ldots,0)
		=
		\widehat{\mu_1}(|\xi|).
		\]
		Consequently, a compactly supported radial distribution with
		one-dimensional marginal $\mu_1$ satisfies
		\eqref{eq:measureFT} whenever $\mu_1$ does.
	
	\noindent\emph{Step 1: the one-dimensional profile.}
	For $\lambda>0$, define
	\begin{equation}\label{eq:k1}
		k_{\lambda,h}^{(1)}(t)
		:=\frac{\sinh(\lambda(h-|t|))}{2\lambda(1+\Cosh)}
		\mathbf1_{\{|t|<h\}},
		\qquad d\mu_{\lambda,h}^{(1)}(t):=k_{\lambda,h}^{(1)}(t)\,dt.
	\end{equation}
	For $\lambda=0$, use its continuous limit
	\begin{equation}\label{eq:k10}
		k_{0,h}^{(1)}(t)=\frac{(h-|t|)_+}{4}.
	\end{equation}
	These are finite positive even densities. 
	Taking Fourier transforms with the convention $e^{-it\xi}$ proves
	\eqref{eq:measureFT} in dimension one (see Subsection \ref{susec:ideas} for the computations). At $\lambda=\xi=0$, the value follows
	by continuity (or by integration of \eqref{eq:k10}).
	
	\noindent\emph{Step 2: the two-dimensional profile.}
	Define $d\mu_{\lambda,h}^{(2)}(z)=k_{\lambda,h}^{(2)}(|z|)\,dz$, where
	\begin{equation}\label{eq:k2}
		k_{\lambda,h}^{(2)}(r)
		:=\frac1{2\pi(1+\Cosh)}
		\int_r^h\frac{\cosh(\lambda(h-s))}{\sqrt{s^2-r^2}}\,ds,
		\qquad 0<r<h,
	\end{equation}
	and the density is zero for $r\ge h$. It is strictly positive on $(0,h)$.
	For $|t|<h$, Tonelli's theorem and the change of variables
	$r^2=t^2+(s^2-t^2)v$ yield
	\begin{align*}
		2\int_{|t|}^h\frac{r k_{\lambda,h}^{(2)}(r)}{\sqrt{r^2-t^2}}\,dr
		&=\frac1{\pi(1+\Cosh)}\int_{|t|}^h\cosh(\lambda(h-s))
		\left[\int_{|t|}^s
		\frac{r\,dr}{\sqrt{s^2-r^2}\sqrt{r^2-t^2}}\right]ds\\
		&=\frac1{2(1+\Cosh)}\int_{|t|}^h\cosh(\lambda(h-s))\,ds
		=k_{\lambda,h}^{(1)}(t).
	\end{align*}
	Thus the one-dimensional marginal of $\mu_{\lambda,h}^{(2)}$
	is $\mu_{\lambda,h}^{(1)}$. In particular,
	$\mu_{\lambda,h}^{(2)}$ is finite. Its rotational invariance,
	together with the preceding observation, proves
	\eqref{eq:measureFT} in dimension two.
	At $\lambda=0$,
	\begin{equation*}
		k_{0,h}^{(2)}(r)=\frac1{4\pi}\operatorname{arcosh}(h/r)
		=\frac1{4\pi}\log\frac{h+\sqrt{h^2-r^2}}r.
	\end{equation*}
	
	\noindent\emph{Step 3: the Fourier walk and positivity through dimension five.}
	For inverse Fourier transforms of a common radial multiplier, the
	distributional dimension walk is
	\begin{equation}\label{eq:fourierwalk}
		k_{\lambda,h}^{(d+2)}(r)
		=-\frac1{2\pi r}\,\partial_r k_{\lambda,h}^{(d)}(r).
	\end{equation}
	We use the radial distributional interpretation of this identity from \cite[Theorem~4.1]{GT13}.
	Starting from dimensions one and two, this constructs the chains
	$1\to3\to5\to7$ and $2\to4\to6$.
	At each step the resulting distribution is orthogonally invariant, has the
	Fourier transform \eqref{eq:measureFT}, and is supported in $\overline B_h$.
	The last assertion also follows from the radial test identity below, since
	differentiation does not enlarge the radial support.
	Uniqueness follows from injectivity of the Fourier transform on tempered
	distributions.
	
	The first odd lift gives
	\begin{equation*}
		k_{\lambda,h}^{(3)}(r)
		=\frac{\cosh(\lambda(h-r))}{4\pi(1+\Cosh)r}
		\mathbf1_{\{0<r<h\}}.
	\end{equation*}
	The profile in dimension one is continuous at $h$, so this step has no
	surface term. In the next lift, its three-dimensional successor has a jump
	at $h$, and thus its distributional derivative gives
	\begin{equation*}
			d\mu_{\lambda,h}^{(5)}(r)
			=\Alh\left[
			\frac{\lambda\sinh(\lambda(h-r))}{r^2}
			+\frac{\cosh(\lambda(h-r))}{r^3}
			\right]\mathbf1_{\{0<r<h\}}\,dz +\frac{\Alh}{h^2}\,d\HH^4\llcorner\partial B_h.
	\end{equation*}
	Both measures are positive, with strictly positive absolutely continuous densities. Their
	origin singularities are respectively $O(r^{-1})$ and $O(r^{-3})$, which
	are integrable in their respective dimensions; the surface part is finite.
	
	For the even lift, integration by parts in \eqref{eq:k2} gives
	\[
	k_{\lambda,h}^{(2)}(r)=\frac1{2\pi(1+\Cosh)}
	\left[\operatorname{arcosh}(h/r)
	+\lambda\int_r^h\sinh(\lambda(h-s))\operatorname{arcosh}(s/r)\,ds\right].
	\]
	Differentiation on $(0,h)$ and \eqref{eq:fourierwalk} yield
	\begin{equation}\label{eq:k4}
		k_{\lambda,h}^{(4)}(r)
		=\frac1{4\pi^2(1+\Cosh)r^2}
		\left[\frac h{\sqrt{h^2-r^2}}
		+\lambda\int_r^h
		\frac{s\sinh(\lambda(h-s))}{\sqrt{s^2-r^2}}\,ds\right],
		\qquad 0<r<h.
	\end{equation}
	The two-dimensional profile tends to zero at $h$, and its derivative is
	integrable near $h$, so this last lift has no surface term.
	The expression in \eqref{eq:k4} is strictly positive, is $O(r^{-2})$ near
	zero, and is $O((h-r)^{-1/2})$ near $h$. Thus it defines a finite positive
	measure in dimension four. These calculations prove all the positivity and
	density assertions through dimension five.
	
	\noindent\emph{Step 4: radial pairings in dimensions six and seven.}
	For a smooth compactly
	supported radial test $\psi(z)=\psi_0(|z|)$, the weak form of
	\eqref{eq:fourierwalk} is
	\begin{align}
		\big\langle\mu_{\lambda,h}^{(d+2)},\psi\big\rangle
		=-\frac{\omega_{d+1}}{2\pi}
		\big\langle \partial_r k_{\lambda,h}^{(d)},r^d\psi_0\big\rangle
		=\frac{\omega_{d+1}}{2\pi}
		\big\langle k_{\lambda,h}^{(d)},(r^d\psi_0)'\big\rangle =\left\langle\mu_{\lambda,h}^{(d)},
		\psi_0(|\cdot|)+\frac{|\cdot|}{d}\psi_0'(|\cdot|)\right\rangle,
		\label{eq:weakfourierwalk}
	\end{align}
	where we used $\omega_{d+1}=2\pi\omega_{d-1}/d$.
	For $d=4,5$, the second equality above is justified using a cut-off approximation on $(\varepsilon,2\varepsilon)$:
	since $k_{\lambda,h}^{(d)}(r)=O(r^{2-d})$ close to zero, the error is bounded by
	\[
	\frac C\varepsilon\int_\varepsilon^{2\varepsilon}
	r^d|k_{\lambda,h}^{(d)}(r)|\,|\psi_0(r)|\,dr
	\le C\varepsilon^2\|\psi_0\|_\infty.
	\]
	Consequently \eqref{eq:weakfourierwalk} gives
	\eqref{eq:liftedradialpairing}. The resulting expression is bounded by a constant
	times the $C^1$ norm of the test function on a fixed neighborhood of $\overline B_h$. This implies that these distributions have order at most one.
	
	To extend the radial pairing, let $\psi_0$ belong to the class in the
	statement. Its derivative is bounded, and $r\psi_0'(r)\to0$ as $r\to0^+$.
	Choose smooth cut-off functions $\chi_\varepsilon$ equal to zero on $[0,\varepsilon]$
	and to one on $[2\varepsilon,\infty)$, with
	$|\chi_\varepsilon'|\le C/\varepsilon$. The profiles
	\[
	\psi_0(0)+\chi_\varepsilon(r)(\psi_0(r)-\psi_0(0))
	\]
	converge to $\psi_0$ in the norm
	$\|\psi_0\|_\infty+\|r\psi_0'\|_\infty$ on a neighborhood of $[0,h]$.
	They are constant near zero and may then be approximated smoothly away
	from zero in the same norm. A fixed outer cut-off makes the corresponding
	radial tests compactly supported. Since $\mu_{\lambda,h}^{(d)}$ is finite for $d=4,5$,
	we get convergence of their pairings to the right-hand side of
	\eqref{eq:liftedradialpairing} establishing the assertion.

	\noindent\emph{Step 5: mass and first absolute directional moments.}
	Recall that for $1\le d\le5$, the line marginal of $\mu_{\lambda,h}^{(d)}$ is given by $\mu_{\lambda,h}^{(1)}$.
	Rotational invariance makes this true in every direction. For $\lambda>0$,
	integration of \eqref{eq:k1} therefore gives
	\begin{align*}
		\mu_{\lambda,h}^{(d)}(\R^d)
		&=\frac1{\lambda(1+\Cosh)}\int_0^h\sinh(\lambda(h-t))\,dt
		=\frac{\Cosh-1}{\lambda^2(1+\Cosh)}=\sigma_\lambda(h),
		\\
		\int_{\R^d}|z\cdot\nu|\,d\mu_{\lambda,h}^{(d)}(z)
		&=\frac1{\lambda(1+\Cosh)}\int_0^h t\sinh(\lambda(h-t))\,dt
		=\frac{\sinh(\lambda h)-\lambda h}{\lambda^3(1+\Cosh)}
		=j_\lambda(h).
	\end{align*}
	For $\lambda=0$, \eqref{eq:k10} gives
	\begin{equation*}
		\int_\R k_{0,h}^{(1)}(t)\,dt=\frac{h^2}{4},\qquad
		\int_\R|t|k_{0,h}^{(1)}(t)\,dt=\frac{h^3}{12}.
	\end{equation*}
	
	It remains to check the lifted moments. For $d=4,5$, formula
	\eqref{eq:weakfourierwalk} with a cut-off equal to one near the support
	preserves the mass. Moreover, using the extended radial pairing,
	\[
	\big\langle\mu_{\lambda,h}^{(d+2)},|\cdot|\big\rangle
	=\frac{d+1}{d}\int_{\R^d}|z|\,d\mu_{\lambda,h}^{(d)}(z).
	\]
	Since
	\[
	c_d=\frac1{\omega_{d-1}}\int_{\mathbb S^{d-1}}|\theta_1|\,d\HH^{d-1}(\theta),
	\qquad c_{d+2}=\frac d{d+1}c_d,
	\]
	radiality and \eqref{eq:measuremoments} imply
	\[
	c_{d+2}\big\langle\mu_{\lambda,h}^{(d+2)},|\cdot|\big\rangle
	=c_d\int_{\R^d}|z|\,d\mu_{\lambda,h}^{(d)}(z)=j_\lambda(h).
	\]
	This completes the proof.
\end{proof}

\begin{figure}[H]
\centering
\begingroup
\pgfplotsset{compat=1.18}
\definecolor{measurezero}{RGB}{35,79,112}
\definecolor{measurescreened}{RGB}{153,67,36}
\begin{tikzpicture}
\begin{groupplot}[
  group style={group size=3 by 2,horizontal sep=1.13cm,vertical sep=1.55cm},
  width=4.05cm,height=3.25cm,scale only axis,
  xmin=0,xmax=1.23,ymin=0,
  axis lines=left,axis line style={thin},
  xlabel={$r$},ylabel={$k_{\lambda,1}^{(d)}(r)$},
  xtick={0,0.5,1},
  tick align=outside,tick style={thin},
  tick label style={font=\scriptsize},
  label style={font=\footnotesize},
  title style={font=\small,yshift=2pt},
  scaled ticks=false,
  ymajorgrids=true,grid style={gray!18,thin},
  clip=false,
]
\nextgroupplot[title={$d=1,\quad\lambda=0$},ymax=0.28,ytick={0,0.1,0.2}]
\addplot[measurezero,line width=1pt,line cap=round,line join=round] coordinates {
    (0.000000000,0.250000000) (1.000000000,0.000000000)
};
\addplot[measurezero,line width=1pt] coordinates {(1,0) (1.2,0)};
\nextgroupplot[title={$d=2,\quad\lambda=0$},ymax=0.36,ytick={0,0.1,0.2,0.3}]
\addplot[measurezero,line width=1pt,line cap=round,line join=round,<-] coordinates {
    (0.028452618,0.338400000) (0.030347346,0.333267511) (0.032368249,0.328134717) (0.034523728,0.323001575)
    (0.036822746,0.317868038) (0.039274860,0.312734050) (0.041890267,0.307599551) (0.044679840,0.302464469)
    (0.047655178,0.297328724) (0.050828650,0.292192223) (0.054213452,0.287054864) (0.057823655,0.281916528)
    (0.061674270,0.276777078) (0.065781308,0.271636361) (0.070161842,0.266494202) (0.074834087,0.261350400)
    (0.079817468,0.256204727) (0.085132704,0.251056924) (0.090801895,0.245906694) (0.096848610,0.240753699)
    (0.103297991,0.235597554) (0.110176851,0.230437816) (0.117513791,0.225273984) (0.125339316,0.220105481)
    (0.133685961,0.214931649) (0.142588429,0.209751735) (0.152083733,0.204564875) (0.162211353,0.199370081)
    (0.173013395,0.194166214) (0.184534771,0.188951968) (0.196823383,0.183725840) (0.209930323,0.178486095)
    (0.223910085,0.173230728) (0.238820794,0.167957417) (0.254724442,0.162663469) (0.271687152,0.157345745)
    (0.289779449,0.152000579) (0.309076556,0.146623671) (0.329658703,0.141209955) (0.351611465,0.135753430)
    (0.375026114,0.130246947) (0.400000000,0.124681922) (0.410555556,0.122414649) (0.421111111,0.120193386)
    (0.431666667,0.118015381) (0.442222222,0.115878059) (0.452777778,0.113779012) (0.463333333,0.111715973)
    (0.473888889,0.109686810) (0.484444444,0.107689507) (0.495000000,0.105722155) (0.505555556,0.103782940)
    (0.516111111,0.101870134) (0.526666667,0.099982085) (0.537222222,0.098117211) (0.547777778,0.096273989)
    (0.558333333,0.094450948) (0.568888889,0.092646667) (0.579444444,0.090859761) (0.590000000,0.089088880)
    (0.600555556,0.087332703) (0.611111111,0.085589929) (0.621666667,0.083859271) (0.632222222,0.082139453)
    (0.642777778,0.080429202) (0.653333333,0.078727240) (0.663888889,0.077032282) (0.674444444,0.075343021)
    (0.685000000,0.073658128) (0.695555556,0.071976239) (0.706111111,0.070295948) (0.716666667,0.068615791)
    (0.727222222,0.066934243) (0.737777778,0.065249695) (0.748333333,0.063560444) (0.758888889,0.061864671)
    (0.769444444,0.060160422) (0.780000000,0.058445575) (0.790555556,0.056717816) (0.801111111,0.054974593)
    (0.811666667,0.053213069) (0.822222222,0.051430063) (0.832777778,0.049621968) (0.843333333,0.047784655)
    (0.853888889,0.045913341) (0.864444444,0.044002411) (0.875000000,0.042045184) (0.885555556,0.040033583)
    (0.896111111,0.037957660) (0.906666667,0.035804904) (0.917222222,0.033559175) (0.927777778,0.031198999)
    (0.938333333,0.028694687) (0.948888889,0.026003073) (0.959444444,0.023056974) (0.970000000,0.019740896)
    (0.978409430,0.016687123) (0.984461576,0.014120161) (0.988817219,0.011956762) (0.991951913,0.010130082)
    (0.994207907,0.008585660) (0.995831514,0.007278636) (0.997000000,0.006171764) (0.997840943,0.005233934)
    (0.998446158,0.004439049) (0.998881722,0.003765152) (0.999195191,0.003193722) (0.999420791,0.002709117)
    (0.999583151,0.002298105) (0.999700000,0.001949486) (0.999784094,0.001653774) (0.999844616,0.001402932)
    (0.999888172,0.001190146) (0.999919519,0.001009639) (0.999942079,0.000856512) (0.999958315,0.000726611)
    (0.999970000,0.000616412) (1.000000000,0.000000000)
};
\addplot[measurezero,line width=1pt] coordinates {(1,0) (1.2,0)};
\nextgroupplot[title={$d=3,\quad\lambda=0$},ymax=0.4,ytick={0,0.1,0.2,0.3,0.4}]
\addplot[measurezero,line width=1pt,line cap=round,line join=round,<-] coordinates {
    (0.105821106,0.376000000) (0.109309362,0.364001170) (0.112912605,0.352385243) (0.116634624,0.341140002)
    (0.120479334,0.330253616) (0.124450780,0.319714635) (0.128553140,0.309511971) (0.132790729,0.299634892)
    (0.137168005,0.290073008) (0.141689572,0.280816261) (0.146360186,0.271854914) (0.151184761,0.263179539)
    (0.156168373,0.254781011) (0.161316262,0.246650494) (0.166633845,0.238779437) (0.172126716,0.231159560)
    (0.177800651,0.223782846) (0.183661621,0.216641536) (0.189715790,0.209728119) (0.195969527,0.203035321)
    (0.202429411,0.196556102) (0.209102236,0.190283646) (0.215995022,0.184211356) (0.223115020,0.178332843)
    (0.230469720,0.172641924) (0.238066857,0.167132612) (0.245914425,0.161799113) (0.254020678,0.156635814)
    (0.262394143,0.151637286) (0.271043628,0.146798270) (0.279978233,0.142113676) (0.289207355,0.137578575)
    (0.298740704,0.133188197) (0.308588307,0.128937924) (0.318760523,0.124823285) (0.329268054,0.120839952)
    (0.340121952,0.116983734) (0.351333634,0.113250574) (0.362914896,0.109636546) (0.374877918,0.106137849)
    (0.387235287,0.102750801) (0.400000000,0.099471839) (0.410555556,0.096914377) (0.421111111,0.094485125)
    (0.431666667,0.092174677) (0.442222222,0.089974528) (0.452777778,0.087876962) (0.463333333,0.085874969)
    (0.473888889,0.083962162) (0.484444444,0.082132711) (0.495000000,0.080381284) (0.505555556,0.078702994)
    (0.516111111,0.077093352) (0.526666667,0.075548232) (0.537222222,0.074063831) (0.547777778,0.072636637)
    (0.558333333,0.071263407) (0.568888889,0.069941137) (0.579444444,0.068667042) (0.590000000,0.067438535)
    (0.600555556,0.066253214) (0.611111111,0.065108840) (0.621666667,0.064003328) (0.632222222,0.062934731)
    (0.642777778,0.061901231) (0.653333333,0.060901126) (0.663888889,0.059932824) (0.674444444,0.058994831)
    (0.685000000,0.058085746) (0.695555556,0.057204253) (0.706111111,0.056349114) (0.716666667,0.055519166)
    (0.727222222,0.054713311) (0.737777778,0.053930515) (0.748333333,0.053169803) (0.758888889,0.052430252)
    (0.769444444,0.051710992) (0.780000000,0.051011200) (0.790555556,0.050330094) (0.801111111,0.049666938)
    (0.811666667,0.049021030) (0.822222222,0.048391706) (0.832777778,0.047778335) (0.843333333,0.047180319)
    (0.853888889,0.046597088) (0.864444444,0.046028101) (0.875000000,0.045472841) (0.885555556,0.044930818)
    (0.896111111,0.044401565) (0.906666667,0.043884635) (0.917222222,0.043379603) (0.927777778,0.042886063)
    (0.938333333,0.042403626) (0.948888889,0.041931923) (0.959444444,0.041470599) (0.970000000,0.041019315)
    (0.978409430,0.040666754) (0.984461576,0.040416748) (0.988817219,0.040238717) (0.991951913,0.040111557)
    (0.994207907,0.040020538) (0.995831514,0.039955289) (0.997000000,0.039908461) (0.997840943,0.039874828)
    (0.998446158,0.039850657) (0.998881722,0.039833280) (0.999195191,0.039820784) (0.999420791,0.039811795)
    (0.999583151,0.039805329) (0.999700000,0.039800676) (0.999784094,0.039797328) (0.999844616,0.039794919)
    (0.999888172,0.039793186) (0.999919519,0.039791938) (0.999942079,0.039791041) (0.999958315,0.039790394)
    (0.999970000,0.039789929) (1.000000000,0.039788736)
};
\addplot[measurezero,line width=1pt] coordinates {(1,0) (1.2,0)};
\draw[measurezero,densely dotted,thin] (axis cs:1,0) -- (axis cs:1,0.039788736);
\addplot[only marks,mark=*,mark size=0.8pt,mark options={draw=measurezero,fill=white}] coordinates {(1,0.039788736)};
\addplot[only marks,mark=*,mark size=0.5pt,measurezero] coordinates {(1,0)};
\nextgroupplot[title={$d=1,\quad\lambda=3$},ymax=0.28,ytick={0,0.1,0.2}]
\addplot[measurescreened,line width=1pt,line cap=round,line join=round] coordinates {
    (0.000000000,0.150858042) (0.012500000,0.145277502) (0.025000000,0.139901283) (0.037500000,0.134721823)
    (0.050000000,0.129731837) (0.062500000,0.124924309) (0.075000000,0.120292476) (0.087500000,0.115829823)
    (0.100000000,0.111530076) (0.112500000,0.107387186) (0.125000000,0.103395327) (0.137500000,0.099548885)
    (0.150000000,0.095842450) (0.162500000,0.092270809) (0.175000000,0.088828939) (0.187500000,0.085512000)
    (0.200000000,0.082315326) (0.212500000,0.079234421) (0.225000000,0.076264953) (0.237500000,0.073402745)
    (0.250000000,0.070643771) (0.262500000,0.067984153) (0.275000000,0.065420148) (0.287500000,0.062948150)
    (0.300000000,0.060564685) (0.312500000,0.058266398) (0.325000000,0.056050058) (0.337500000,0.053912547)
    (0.350000000,0.051850860) (0.362500000,0.049862097) (0.375000000,0.047943460) (0.387500000,0.046092252)
    (0.400000000,0.044305869) (0.412500000,0.042581798) (0.425000000,0.040917615) (0.437500000,0.039310979)
    (0.450000000,0.037759630) (0.462500000,0.036261388) (0.475000000,0.034814143) (0.487500000,0.033415862)
    (0.500000000,0.032064578) (0.512500000,0.030758389) (0.525000000,0.029495460) (0.537500000,0.028274013)
    (0.550000000,0.027092332) (0.562500000,0.025948753) (0.575000000,0.024841670) (0.587500000,0.023769524)
    (0.600000000,0.022730807) (0.612500000,0.021724060) (0.625000000,0.020747866) (0.637500000,0.019800852)
    (0.650000000,0.018881686) (0.662500000,0.017989075) (0.675000000,0.017121765) (0.687500000,0.016278535)
    (0.700000000,0.015458199) (0.712500000,0.014659604) (0.725000000,0.013881626) (0.737500000,0.013123172)
    (0.750000000,0.012383174) (0.762500000,0.011660593) (0.775000000,0.010954411) (0.787500000,0.010263635)
    (0.800000000,0.009587294) (0.812500000,0.008924437) (0.825000000,0.008274131) (0.837500000,0.007635463)
    (0.850000000,0.007007533) (0.862500000,0.006389458) (0.875000000,0.005780370) (0.887500000,0.005179411)
    (0.900000000,0.004585737) (0.912500000,0.003998512) (0.925000000,0.003416910) (0.937500000,0.002840115)
    (0.950000000,0.002267313) (0.962500000,0.001697701) (0.975000000,0.001130476) (0.987500000,0.000564841)
    (1.000000000,0.000000000)
};
\addplot[measurescreened,line width=1pt] coordinates {(1,0) (1.2,0)};
\nextgroupplot[title={$d=2,\quad\lambda=3$},ymax=0.36,ytick={0,0.1,0.2,0.3}]
\addplot[measurescreened,line width=1pt,line cap=round,line join=round,<-] coordinates {
    (0.036943496,0.338400000) (0.039153468,0.330141642) (0.041495641,0.321898594) (0.043977924,0.313672311)
    (0.046608698,0.305464377) (0.049396845,0.297276508) (0.052351780,0.289110570) (0.055483481,0.280968584)
    (0.058802521,0.272852740) (0.062320107,0.264765404) (0.066048116,0.256709136) (0.069999135,0.248686693)
    (0.074186506,0.240701048) (0.078624367,0.232755399) (0.083327702,0.224853179) (0.088312392,0.216998071)
    (0.093595268,0.209194018) (0.099194168,0.201445229) (0.105127995,0.193756196) (0.111416786,0.186131697)
    (0.118081774,0.178576804) (0.125145464,0.171096890) (0.132631706,0.163697628) (0.140565778,0.156384998)
    (0.148974469,0.149165276) (0.157886170,0.142045036) (0.167330972,0.135031134) (0.177340765,0.128130696)
    (0.187949347,0.121351096) (0.199192538,0.114699932) (0.211108300,0.108184993) (0.223736867,0.101814222)
    (0.237120879,0.095595666) (0.251305527,0.089537424) (0.266338705,0.083647583) (0.282271173,0.077934146)
    (0.299156726,0.072404953) (0.317052379,0.067067588) (0.336018555,0.061929279) (0.356119294,0.056996787)
    (0.377422466,0.052276289) (0.400000000,0.047773239) (0.410555556,0.045826310) (0.421111111,0.043972097)
    (0.431666667,0.042205225) (0.442222222,0.040520699) (0.452777778,0.038913876) (0.463333333,0.037380427)
    (0.473888889,0.035916312) (0.484444444,0.034517755) (0.495000000,0.033181222) (0.505555556,0.031903400)
    (0.516111111,0.030681180) (0.526666667,0.029511640) (0.537222222,0.028392029) (0.547777778,0.027319756)
    (0.558333333,0.026292376) (0.568888889,0.025307579) (0.579444444,0.024363180) (0.590000000,0.023457111)
    (0.600555556,0.022587407) (0.611111111,0.021752204) (0.621666667,0.020949728) (0.632222222,0.020178291)
    (0.642777778,0.019436278) (0.653333333,0.018722149) (0.663888889,0.018034428) (0.674444444,0.017371696)
    (0.685000000,0.016732592) (0.695555556,0.016115799) (0.706111111,0.015520048) (0.716666667,0.014944105)
    (0.727222222,0.014386769) (0.737777778,0.013846868) (0.748333333,0.013323250) (0.758888889,0.012814779)
    (0.769444444,0.012320330) (0.780000000,0.011838778) (0.790555556,0.011368993) (0.801111111,0.010909828)
    (0.811666667,0.010460113) (0.822222222,0.010018635) (0.832777778,0.009584127) (0.843333333,0.009155248)
    (0.853888889,0.008730553) (0.864444444,0.008308463) (0.875000000,0.007887219) (0.885555556,0.007464821)
    (0.896111111,0.007038938) (0.906666667,0.006606781) (0.917222222,0.006164909) (0.927777778,0.005708913)
    (0.938333333,0.005232894) (0.948888889,0.004728507) (0.959444444,0.004183044) (0.970000000,0.003575032)
    (0.978409430,0.003018852) (0.984461576,0.002553087) (0.988817219,0.002161315) (0.991951913,0.001830858)
    (0.994207907,0.001551611) (0.995831514,0.001315353) (0.997000000,0.001115303) (0.997840943,0.000945817)
    (0.998446158,0.000802170) (0.998881722,0.000680390) (0.999195191,0.000577128) (0.999420791,0.000489556)
    (0.999583151,0.000415283) (0.999700000,0.000352285) (0.999784094,0.000298848) (0.999844616,0.000253519)
    (0.999888172,0.000215067) (0.999919519,0.000182448) (0.999942079,0.000154777) (0.999958315,0.000131303)
    (0.999970000,0.000111390) (1.000000000,0.000000000)
};
\addplot[measurescreened,line width=1pt] coordinates {(1,0) (1.2,0)};
\nextgroupplot[title={$d=3,\quad\lambda=3$},ymax=0.4,ytick={0,0.1,0.2,0.3,0.4}]
\addplot[measurescreened,line width=1pt,line cap=round,line join=round,<-] coordinates {
    (0.130524286,0.376000000) (0.134138664,0.361965722) (0.137853129,0.348352743) (0.141670452,0.335150418)
    (0.145593481,0.322348389) (0.149625143,0.309936575) (0.153768447,0.297905162) (0.158026485,0.286244596)
    (0.162402433,0.274945575) (0.166899556,0.263999043) (0.171521210,0.253396179) (0.176270843,0.243128391)
    (0.181151999,0.233187310) (0.186168321,0.223564781) (0.191323551,0.214252858) (0.196621535,0.205243796)
    (0.202066227,0.196530043) (0.207661690,0.188104237) (0.213412097,0.179959197) (0.219321741,0.172087916)
    (0.225395029,0.164483560) (0.231636495,0.157139454) (0.238050794,0.150049085) (0.244642714,0.143206089)
    (0.251417171,0.136604250) (0.258379222,0.130237494) (0.265534060,0.124099883) (0.272887025,0.118185609)
    (0.280443602,0.112488991) (0.288209431,0.107004469) (0.296190304,0.101726603) (0.304392178,0.096650061)
    (0.312821172,0.091769623) (0.321483575,0.087080171) (0.330385850,0.082576690) (0.339534641,0.078254258)
    (0.348936772,0.074108049) (0.358599261,0.070133326) (0.368529315,0.066325438) (0.378734345,0.062679818)
    (0.389221965,0.059191982) (0.400000000,0.055857519) (0.410555556,0.052816742) (0.421111111,0.049980045)
    (0.431666667,0.047330977) (0.442222222,0.044854706) (0.452777778,0.042537827) (0.463333333,0.040368201)
    (0.473888889,0.038334811) (0.484444444,0.036427638) (0.495000000,0.034637557) (0.505555556,0.032956238)
    (0.516111111,0.031376067) (0.526666667,0.029890071) (0.537222222,0.028491857) (0.547777778,0.027175551)
    (0.558333333,0.025935755) (0.568888889,0.024767494) (0.579444444,0.023666183) (0.590000000,0.022627588)
    (0.600555556,0.021647795) (0.611111111,0.020723185) (0.621666667,0.019850403) (0.632222222,0.019026339)
    (0.642777778,0.018248109) (0.653333333,0.017513029) (0.663888889,0.016818609) (0.674444444,0.016162528)
    (0.685000000,0.015542626) (0.695555556,0.014956889) (0.706111111,0.014403440) (0.716666667,0.013880525)
    (0.727222222,0.013386510) (0.737777778,0.012919862) (0.748333333,0.012479154) (0.758888889,0.012063046)
    (0.769444444,0.011670286) (0.780000000,0.011299702) (0.790555556,0.010950195) (0.801111111,0.010620735)
    (0.811666667,0.010310357) (0.822222222,0.010018155) (0.832777778,0.009743279) (0.843333333,0.009484934)
    (0.853888889,0.009242370) (0.864444444,0.009014886) (0.875000000,0.008801822) (0.885555556,0.008602559)
    (0.896111111,0.008416516) (0.906666667,0.008243147) (0.917222222,0.008081940) (0.927777778,0.007932414)
    (0.938333333,0.007794118) (0.948888889,0.007666628) (0.959444444,0.007549547) (0.970000000,0.007442503)
    (0.978409430,0.007364173) (0.984461576,0.007311511) (0.988817219,0.007275496) (0.991951913,0.007250537)
    (0.994207907,0.007233069) (0.995831514,0.007220751) (0.997000000,0.007212016) (0.997840943,0.007205797)
    (0.998446158,0.007201357) (0.998881722,0.007198179) (0.999195191,0.007195901) (0.999420791,0.007194267)
    (0.999583151,0.007193093) (0.999700000,0.007192249) (0.999784094,0.007191643) (0.999844616,0.007191207)
    (0.999888172,0.007190893) (0.999919519,0.007190668) (0.999942079,0.007190505) (0.999958315,0.007190388)
    (0.999970000,0.007190304) (1.000000000,0.007190089)
};
\addplot[measurescreened,line width=1pt] coordinates {(1,0) (1.2,0)};
\draw[measurescreened,densely dotted,thin] (axis cs:1,0) -- (axis cs:1,0.007190089);
\addplot[only marks,mark=*,mark size=0.8pt,mark options={draw=measurescreened,fill=white}] coordinates {(1,0.007190089)};
\addplot[only marks,mark=*,mark size=0.5pt,measurescreened] coordinates {(1,0)};
\end{groupplot}
\end{tikzpicture}
\endgroup
\caption{Radial profiles of the Lebesgue densities
$k_{\lambda,1}^{(d)}(r)$ of the comparison measures in
Proposition~\ref{prop:measure}, with $h=1$ and $r=|z|$.
The first row corresponds to $\lambda=0$ and the second to $\lambda=3$;
the columns correspond to $d=1,2,3$.
Arrows indicate divergence as $r \to 0^+$ for $d=2,3$.
All densities vanish for $r>1$. For $d=3$, the open circle marks
the positive left limit at $r=1$.}
\label{fig:comparison-measure-densities}
\end{figure}
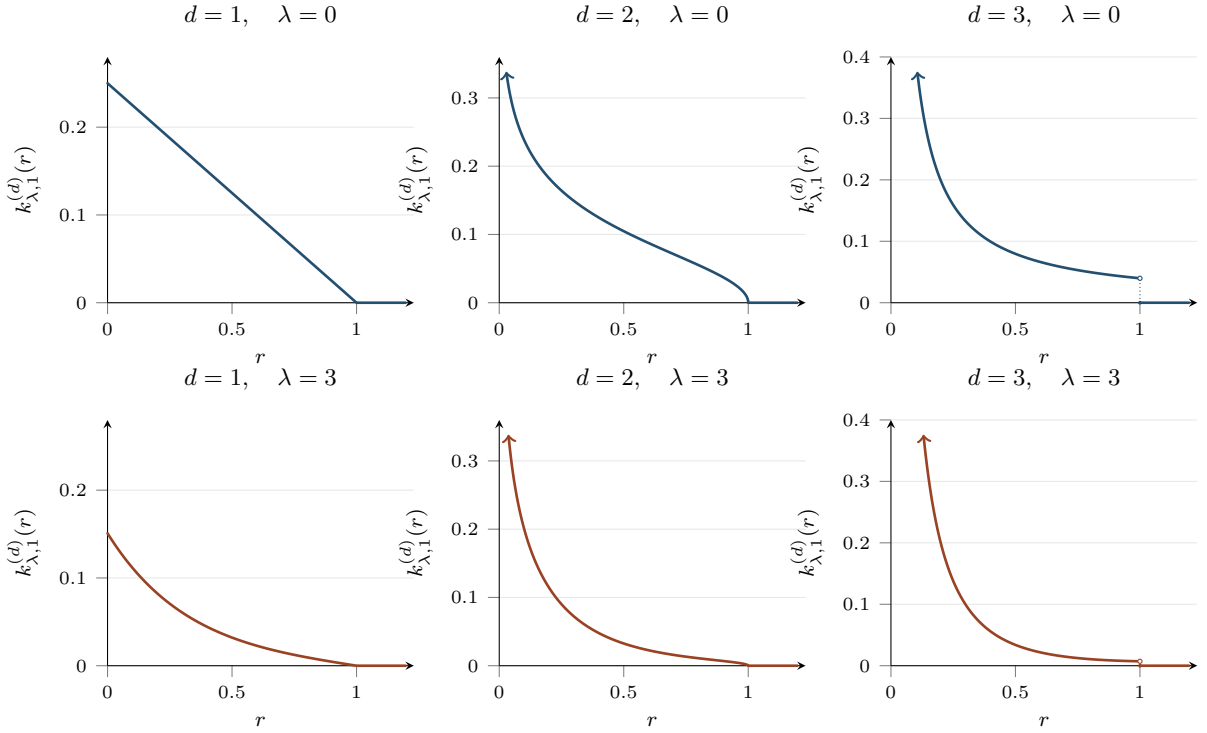

\section{The sharp estimate, rigidity, and proofs of the theorems}\label{sec:proofs}
In this section we prove the additional intermediate results needed to prove Theorems \ref{thm:periodic} and \ref{thm:neumann}.
\subsection{The directional translation inequality}
We now state and prove a general result for $BV$ function on the torus. A similar statement also appears in \cite[Lemma~3.1]{Mil26} and \cite[Proposition~11]{Gal11}.

\begin{lemma}[Translation control and radial geometric defect]
	\label{lem:translation}
	Let $T$ be a periodic torus and let $U\in BV(T)$.
	Write $DU=\theta|DU|$ and define $\trn_U,\varn_U$
	as in \eqref{eq:translationdefs}. Then, for every $z\in\R^d$, we have
	\begin{equation}\label{eq:translation}
		\trn_U(z)\le\varn_U(z).
	\end{equation}
	In particular, the function $\defect_U(z):=\varn_U(z)-\trn_U(z)$
	is nonnegative and Lipschitz on $\R^d$.
	For every $e\in\mathbb S^{d-1}$, the map
	$r\mapsto\defect_U(re)$ is nondecreasing on $[0,\infty)$.
	Consequently, its spherical average
	\[
	\psi_U(r):=
	\fint_{\mathbb S^{d-1}}\defect_U(re)\,d\HH^{d-1}(e)
	\]
	is nonnegative, nondecreasing and Lipschitz on $[0,\infty)$,
	with $\psi_U(0)=0$ and
	\[
	\operatorname{Lip}(\psi_U)\le 2c_d|DU|(T),
	\qquad
	c_d:=\fint_{\mathbb S^{d-1}}|e_1|\,d\HH^{d-1}(e).
	\]
	Finally, if $d \geq 2$ and $U \in BV(T;\{-1,1\})$, then $\psi_U\in C^1((0,\infty))$.
\end{lemma}

\begin{proof}
	In the proof, periodic functions and measures are identified with their periodic extensions to $\mathbb R^d$.
	For smooth periodic $U$, the fundamental theorem of calculus gives
	\[
	|U(x+z)-U(x)|
	\le\int_0^1|z\cdot\nabla U(x+tz)|\,dt.
	\]
	Integration over the torus and translation invariance prove
	\eqref{eq:translation}.
	Given a general $BV$ function, let
	$U_\varepsilon=\rho_\varepsilon*U$ be a periodic mollification.
	Then
	\[
	\int_T|z\cdot\nabla U_\varepsilon|\,dx
	\le |z\cdot DU|(T)=\varn_U(z),
	\]
	whereas $|\trn_{U_\varepsilon}(z)-\trn_U(z)|\le2\|U_\varepsilon-U\|_{L^1(T)}$.
	Passing to the limit proves \eqref{eq:translation}.
	
	The triangle inequality and translation invariance imply
	\[
	|\trn_U(z)-\trn_U(z')|
	\le\trn_U(z-z')
	\le\varn_U(z-z')
	\le |z-z'|\,|DU|(T).
	\]
	Together with
	\[
	|\varn_U(z)-\varn_U(z')|
	\le |z-z'|\,|DU|(T),
	\]
	this proves that $\defect_U$ is Lipschitz.
	
	Fix $e\in\mathbb S^{d-1}$ and $0\le r\le s$.
	The same argument gives
	\[
	\trn_U(se)-\trn_U(re)
	\le\trn_U((s-r)e)
	\le(s-r)|D_eU|(T).
	\]
	Since $\varn_U(re)=r|D_eU|(T)$, it follows that
	\[
	0\le\defect_U(se)-\defect_U(re)
	\le2(s-r)|D_eU|(T).
	\]
	Averaging over $e$ proves monotonicity of $\psi_U$ and the
	Lipschitz bound. Indeed, the polar decomposition and
	rotational invariance give
	\[
	\fint_{\mathbb S^{d-1}}|D_eU|(T)\,d\HH^{d-1}(e)
	=
	\int_T
	\fint_{\mathbb S^{d-1}}|e\cdot\theta(x)|
	\,d\HH^{d-1}(e)\,d|DU|(x)
	=c_d|DU|(T).
	\]
	
	It remains to prove $C^1$ regularity of $\psi_U$ when $U \in BV(T;\{-1,1\})$. For $z \in \R^d$ and $r>0$ set
	\[
	\mathcal{C}(z):=\int_T U(x+z)U(x) \, dx, \qquad \mbox{and} \qquad \overline{\mathcal{C}}(r):= \fint_{\mathbb S^{d-1}} \mathcal{C}(re) \, d\HH^{d-1}(e).
	\]
	Observe that using the identity $|U(x+z)-U(x)|=1-U(x+z)U(x)$ we get for every $r>0$
	\[
	\psi_U(r)=c_d r |DU|(T)-|T|+\overline{\mathcal{C}}(r).
	\]
	Thus, it is enough to prove that $\overline{\mathcal{C}}'(r) \in C((0,\infty))$. Assume temporarily that $\mathcal{C}$ is smooth. Then, using the divergence theorem we would have
	\begin{equation}\label{eq:der}
		\frac{d}{dr} \fint_{\mathbb S^{d-1}} \mathcal{C}(re) \, d\HH^{d-1}(e)=\frac{1}{\omega_{d-1}r^{d-1}} \int_{\partial B_r} \nabla \mathcal{C} \cdot \nu_{B_r} \, d\HH^{d-1}(x)=\frac{1}{\omega_{d-1}r^{d-1}} \int_{B_r} \Delta \mathcal{C}(z) \, dz.
	\end{equation}
	That is,
	\[
	\overline{\mathcal{C}}'(r)=\frac{(\Delta \mathcal{C})(B_r)}{\omega_{d-1}r^{d-1}} \qquad \mbox{for $r>0$}.
	\]
	Therefore, if we prove that $r \mapsto (\Delta \mathcal{C})(B_r)$ is continuous on $(0,\infty)$, we can justify equality \eqref{eq:der} using approximations with a mollifier and the identity above would imply the continuity of $\overline{\mathcal{C}}'$. 
	To this aim, let us set $\check{U}(x):=U(-x)$. We have $\mathcal{C}=U * \check{U}$ and thus $\Delta \mathcal{C}=\sum_{i=1}^d D_i U * D_i \check{U}$. In particular, $\Delta \mathcal{C}$ is a signed finite measure on $T$.
	Moreover, being $U \in BV(T;\{-1,1\})$, we have that $|DU|$ vanishes on sets of zero $\HH^{d-1}$-measure. In particular, for every $r>0$ and $x \neq y$ we get $|DU|(\partial B_r(x) \cap \partial B_r(y))=0$. This means that $|DU|(\partial B_r(x))>0$ only for a countable number of centers $x \in \R^d$, a $|DU|$-null set. Thus, for every $r >0$,
	\[
	\int_T |DU|(\partial B_r(x)) \, d|DU|(x)=0,
	\]
	giving $|\Delta \mathcal{C}|(\partial B_r)=0$. This implies the continuity of $r \mapsto (\Delta \mathcal{C})(B_r)$ on $(0,\infty)$. 
	Let now $\rho_\varepsilon$ be a nonnegative, radial standard
	mollifier and set $\mathcal{C}_\varepsilon:=\rho_\varepsilon*\mathcal{C}$, and $\overline{\mathcal{C}}_\varepsilon$ its spherical average.
	Since $\mathcal{C}$ is Lipschitz, $\overline{\mathcal{C}}_\varepsilon$ converges
	uniformly to $\overline{\mathcal{C}}$. Using \eqref{eq:der} since $\mathcal{C}_\varepsilon$ is smooth,
	\[
	\overline{\mathcal{C}}_\varepsilon'(r)
	=
	\frac{1}{\omega_{d-1}r^{d-1}}
	\int_{B_r}(\rho_\varepsilon*(\Delta \mathcal{C}))(x)\,dx.
	\]
	For every $r>0$, by dominated convergence we have that $(\rho_\varepsilon*(\Delta \mathcal{C}))(B_r)$ converges to $(\Delta \mathcal{C})(B_r)$ since $\partial B_r$ is $|\Delta \mathcal{C}|$-negligible.
	
	Fix $a \in (0,\infty)$. For $r \in (a,\infty)$, integrating from $a$ to $r$ and passing to the limit in $\varepsilon \to 0$ by dominated
	convergence yields
	\[
	\overline{\mathcal{C}}(r)-\overline{\mathcal{C}}(a)
	=
	\int_a^r
	\frac{(\Delta \mathcal{C})(B_t)}{\omega_{d-1}t^{d-1}}\,dt.
	\]
	Hence $\overline{\mathcal{C}}\in C^1((0,\infty))$.
\end{proof}

\subsection{The global comparison}
We are ready to show the sharp global comparison inequality which will entail Theorems \ref{thm:periodic} and \ref{thm:neumann}.

\begin{proposition}[Sharp comparison on the torus]
	\label{prop:comparison}
	Let $1\le d\le7$, $\lambda\ge0$, $h>0$, and let
	$T=\T_\ell^d$ have volume $V=\ell^d$.
	For every $U\in\A_\lambda(T)$,
	\begin{equation}\label{eq:comparison}
		N_\lambda^{\mathrm{per}}(U)+j_\lambda(h)|DU|(T)
		\ge\sigma_\lambda(h)V.
	\end{equation}
	More precisely, let $\mu_{\lambda,h}^{(d)}$ be given by
	Proposition~\ref{prop:measure}, and set
	\[
	\psi_U(r):=
	\fint_{\mathbb S^{d-1}}\defect_U(re)\,d\HH^{d-1}(e).
	\]
	Then,
	\begin{equation}\label{eq:exactdefect}
			N_\lambda^{\mathrm{per}}(U)
			+j_\lambda(h)|DU|(T)-\sigma_\lambda(h)V=
			\big\langle\mu_{\lambda,h}^{(d)},\psi_U(|\cdot|)\big\rangle
			+V\sum_{q\in(2\pi/\ell)\Z^d}
			\frac{1+\cos(h|q|)}
			{(1+\cosh(\lambda h))(\lambda^2+|q|^2)}
			|\widehat U(q)|^2.
	\end{equation}
	For $\lambda=0$, the sum omits $q=0$.
	Both terms on the right-hand side are nonnegative.
	
	For $d\le5$, the first term agrees with
	\[
	\int_{\R^d}\defect_U(z)\,d\mu_{\lambda,h}^{(d)}(z).
	\]
	For $d=6,7$, it is the extended radial pairing defined
	in Proposition~\ref{prop:measure}.
\end{proposition}

\begin{proof}
	Write $\mu:=\mu_{\lambda,h}^{(d)}$ and define
	\[
	Q_\mu(U):=
	V\sum_{q\in(2\pi/\ell)\Z^d}
	\widehat\mu(q)|\widehat U(q)|^2.
	\]
	This series converges absolutely. Indeed,
	$\widehat\mu$ has the same radial profile as the Fourier
	transform of the finite positive one-dimensional measure by \eqref{eq:measureFT},
	and therefore $|\widehat\mu(q)|\le\sigma_\lambda(h)$.
	
	We first identify $Q_\mu(U)$ with a radial pairing.
	Set as before
	\[
	\mathcal C_U(z):=\int_T U(x+z)U(x)\,dx,
	\qquad
	\overline{\mathcal C}_U(r):=
	\fint_{\mathbb S^{d-1}}\mathcal C_U(re)
	\,d\HH^{d-1}(e).
	\]
	Since $U$ takes values in $\{-1,1\}$,
	\begin{equation}\label{eq:binaryidentity}
		\psi_U(r)
		=c_dr|DU|(T)-V+\overline{\mathcal C}_U(r).
	\end{equation}
	Reasoning as in the proof of Lemma \ref{lem:translation} we have that
	$\mathcal C_U$ is Lipschitz. More precisely, for every $e\in\mathbb S^{d-1}$
	\begin{equation}\label{eq:correlationtranslation}
		|\mathcal C_U(z+re)-\mathcal C_U(z)|
		\le |r|\,|D_eU|(T).
	\end{equation}
	Moreover, by Lemma~\ref{lem:translation} and
	\eqref{eq:binaryidentity},
	$\overline{\mathcal C}_U\in C^1((0,\infty))$ when
	$d=6,7$.
	
	Let $\rho_\varepsilon$ be a nonnegative radial smooth
	mollifier on $\R^d$. Set
	$\mathcal C_{U,\varepsilon}:=\rho_\varepsilon*\mathcal C_U$,
	using the periodic extension of $\mathcal C_U$.
	Denote its spherical average again by
	$\overline{\mathcal C}_{U,\varepsilon}$.
	Orthogonal invariance of $\mu$ and the Fourier expansion
	of the smooth periodic function
	$\mathcal C_{U,\varepsilon}$ give
	\[
		\big\langle\mu,
		\overline{\mathcal C}_{U,\varepsilon}(|\cdot|)
		\big\rangle
		=
		\langle\mu,\mathcal C_{U,\varepsilon}\rangle
		=
		V\sum_{q\in(2\pi/\ell)\Z^d}
		\widehat\mu(q)\widehat\rho_\varepsilon(q)
		|\widehat U(q)|^2.
	\]
	Since $|\widehat\rho_\varepsilon|\le1$ and
	$\widehat\rho_\varepsilon(q)\to1$ as $\varepsilon \to 0^+$, dominated convergence
	shows that the right-hand side tends to $Q_\mu(U)$.
	
	We now pass to the limit on the left-hand side.
	If $d\le5$, this follows from uniform convergence of the
	spherical averages and finiteness of $\mu$.
	For $d=6,7$, radial convolution commutes with spherical
	averaging. Consequently, $\overline{\mathcal C}_{U,\varepsilon}$ converges to $\overline{\mathcal C}_U$ in $C^1_{\mathrm{loc}}((0,\infty))$,
	because the radial function
	$z\mapsto\overline{\mathcal C}_U(|z|)$ is $C^1$
	away from the origin.
	Furthermore, convolution preserves
	\eqref{eq:correlationtranslation}, so
	\[
	|\overline{\mathcal C}_{U,\varepsilon}(r)|\le V,
	\qquad
	|\overline{\mathcal C}_{U,\varepsilon}'(r)|
	\le c_d|DU|(T).
	\]
	The extended pairing formula therefore gives
	\[
	\begin{aligned}
		&\big\langle\mu,
		\overline{\mathcal C}_{U,\varepsilon}(|\cdot|)
		\big\rangle
		=
		\int_{\R^{d-2}}
		\left[
		\overline{\mathcal C}_{U,\varepsilon}(|z|)
		+\frac{|z|}{d-2}
		\overline{\mathcal C}_{U,\varepsilon}'(|z|)
		\right]\,d\mu_{\lambda,h}^{(d-2)}(z).
	\end{aligned}
	\]
	The preceding bounds permit dominated convergence against
	the finite measure $\mu_{\lambda,h}^{(d-2)}$.
	We have therefore proved, in every dimension $1 \leq d \leq 7$, that
	\begin{equation}\label{eq:radialcorrelationpairing}
		Q_\mu(U)
		=
		\big\langle\mu,\overline{\mathcal C}_U(|\cdot|)\big\rangle.
	\end{equation}
	By the mass and radial first-moment identities in
	Proposition~\ref{prop:measure},
	\[
	\langle\mu,1\rangle=\sigma_\lambda(h),
	\qquad
	c_d\langle\mu,|\cdot|\rangle=j_\lambda(h).
	\]
	Combining these identities with
	\eqref{eq:binaryidentity} and
	\eqref{eq:radialcorrelationpairing} yields
	\begin{equation}\label{eq:geometricdefect}
		Q_\mu(U)+j_\lambda(h)|DU|(T)-\sigma_\lambda(h)V
		=
		\big\langle\mu,\psi_U(|\cdot|)\big\rangle
		\ge0.
	\end{equation}
	Indeed, for $d\le5$ this follows from positivity of $\mu$
	and nonnegativity of $\psi_U$.
	For $d=6,7$, Lemma~\ref{lem:translation} shows that
	$\psi_U$ is nonnegative, nondecreasing, Lipschitz,
	and $C^1$ away from zero, so the positivity assertion
	of Proposition~\ref{prop:measure} applies.
	For $d\le5$, rotational invariance also gives
	\[
	\big\langle\mu,\psi_U(|\cdot|)\big\rangle
	=\int_{\R^d}\defect_U(z)\,d\mu(z).
	\]
	
	Finally, for $\lambda>0$,
	\[
	N_\lambda^{\mathrm{per}}(U)
	=
	V\sum_{q\in(2\pi/\ell)\Z^d}
	\frac{|\widehat U(q)|^2}{\lambda^2+|q|^2}.
	\]
	For $\lambda=0$, the same formula holds with $q\ne0$,
	and $\widehat U(0)=0$ by the admissibility condition.
	By \eqref{eq:measureFT},
	\begin{equation}\label{eq:spectralgap}
		\frac1{\lambda^2+|q|^2}-\widehat\mu(q)
		=
		\frac{1+\cos(h|q|)}
		{(1+\cosh(\lambda h))(\lambda^2+|q|^2)}
		\ge0,
	\end{equation}
	with $q\ne0$ when $\lambda=0$.
	Subtracting $Q_\mu(U)$ from the Fourier representation of
	$N_\lambda^{\mathrm{per}}(U)$ and using
	\eqref{eq:geometricdefect} proves
	\eqref{eq:exactdefect} and \eqref{eq:comparison}.
\end{proof}

The next result entails that if $\lambda>0$ then equality forces $\mean U=0$.
\begin{corollary}[The screened constant mode]\label{cor:mean}
If $\lambda>0$, then
\begin{equation}\label{eq:meanpenalty}
 N_\lambda^{\mathrm{per}}(U)+j_\lambda(h)|DU|(T)
 -\sigma_\lambda(h)V
 \ge\frac{V}{\lambda^2}\sech^2\left(\frac{\lambda h}2\right)\,|\mean U|^2.
\end{equation}
In particular, equality in \eqref{eq:comparison} forces $\mean U=0$.
\end{corollary}
\begin{proof}
Retaining the $q=0$ term in \eqref{eq:exactdefect}, its coefficient is
\[
 \frac{2}{(1+\cosh(\lambda h))\lambda^2}
 =\lambda^{-2}\sech^2\left(\frac{\lambda h}2\right)>0.
\]
\end{proof}

\subsection{Rigidity of zero translation defect}
In this subsection, we prove that if
$U\in BV(\T_\ell^d;\{-1,1\})$ and its translation defect
$\defect_U$ vanishes in a neighborhood of the origin,
then $U$ is one-dimensional.

\begin{lemma}[Zero translation defect implies parallel layers]\label{lem:rigidity}
Let $U\in BV(\T_\ell^d;\{-1,1\})$ be non-constant up to $\mathcal{L}^d$-null sets, and let $h>0$. Suppose
\begin{equation}\label{eq:zerodefect}
 \defect_U(z)=0\qquad\text{for $\mathcal{L}^d$-almost every }|z|<h.
\end{equation}
For the periodic lift to $\R^d$, let $E_+$ and $E_-$ be the density-one representatives of $\{U=1\}$ and $\{U=-1\}$, respectively. Then $E_+$ and $E_-$ are open, and their common boundary is a locally finite family of complete parallel affine hyperplanes. Consecutive planes have distance at least $h$. In particular,
\begin{equation}\label{eq:onedimensional}
 U(x)=g(\nu\cdot x)\quad\text{almost everywhere}
\end{equation}
for some unit vector $\nu$ and $g \in BV_{\mathrm{loc}}(\R;\{-1,1\})$ whose values alternate across the intervals.
\end{lemma}
\begin{proof}
	Set $T:=\T_\ell^d$, and also denote by $U$ its periodic
	lift to $\R^d$. By continuity of $\defect_U$ given by Lemma \ref{lem:translation}, we get that \eqref{eq:zerodefect} holds for every $|z|<h$.
	
	\noindent
	\emph{Step 1: Short segments contain at most one jump.}
	Fix $e\in\mathbb S^{d-1}$ and $r\in(0,h)$.
	By \eqref{eq:zerodefect},
	\[
	\trn_U(re)=\varn_U(re)
	=r\int_T|e\cdot\theta|\,d|DU|
	=r\,|D_eU|(T).
	\]
	By the slicing theorem for $BV$ functions (see \cite[section 3.11]{AFP}), for $\HH^{d-1}$-almost every
	$y\in e^\perp$, the function
	\[
	f_y(s):=U(y+se)
	\]
	belongs to $BV_{\mathrm{loc}}(\R;\{-1,1\})$.
	It is therefore piecewise constant, with locally finitely
	many jumps, each of size two. Define
	\[
	\Lambda_{e,r}(y+se):=\frac12|Df_y|((s,s+r)).
	\]
	By definition $\Lambda_{e,r}(x)$ counts the jumps along the open segment
	from $x$ to $x+re$, and
	\[
	|U(x+re)-U(x)|\le 2\Lambda_{e,r}(x)
	\qquad\text{for almost every }x.
	\]
	Let $F:=[0,\ell)^d$ be a half-open fundamental cell.
	Slicing and Tonelli's theorem give
	\[
	\begin{aligned}
		2\int_F \Lambda_{e,r}(x)\,dx
		&=\int_{e^\perp}\int_{\R}
		\one_F(y+se)
		\int_{(s,s+r)}d|Df_y|(t)\,ds\,d\HH^{d-1}(y)\\
		&=\int_0^r\int_{\R^d}
		\one_F(z-\tau e)\,d|D_eU|(z)\,d\tau
		=\int_0^r|D_eU|(F+\tau e)\,d\tau\\
		&=r\,|D_eU|(T),
	\end{aligned}
	\]
	where the last equality follows from periodicity, since every
	translate of $F$ is a fundamental cell for the periodic
	measure $|D_eU|$.
	Therefore,
	\[
	\begin{aligned}
		\int_F
		\bigl(2\Lambda_{e,r}(x)-|U(x+re)-U(x)|\bigr)\,dx
		=\varn_U(re)-\trn_U(re)
		=\defect_U(re)=0.
	\end{aligned}
	\]
	Since the integrand is nonnegative, we infer that
	\[
	2\Lambda_{e,r}(x)=|U(x+re)-U(x)|\le2
	\qquad\text{for $\mathcal{L}^d$-almost every }x.
	\]
	Hence $\Lambda_{e,r}\le1$ $\mathcal{L}^d$-almost everywhere. By periodicity and Fubini's theorem, for
	$\HH^{d-1}$-almost every $y\in e^\perp$ and
	$\HH^1$-almost every $s\in\R$, the restriction
	$t\mapsto U(y+te)$ has at most one jump in $(s,s+r)$.
	
	\noindent
	\emph{Step 2: Local half-space structure.}
	Let $B\subset\R^d$ be any open ball with $\diam B<h$,
	and choose $r\in(\diam B,h)$.
	For every direction $e$, Step~1 implies that $\HH^{d-1}$-almost every
	line parallel to $e$ contains at most one jump inside $B$.
	Thus both phases have interval sections in $B$, modulo
	one-dimensional null sets.
	Applying Lemma~\ref{lem:euclidean-monotone} to
	$B\cap\{U=1\}$, we conclude that $U$ is either constant
	almost everywhere in $B$, or
	\[
	U=2\one_H-1\qquad\text{almost everywhere in }B
	\]
	for a half-space $H$ whose boundary meets $B$.
	In the nonconstant case, the density-one representatives
	satisfy
	\[
	E_+\cap B=B\cap\operatorname{int}H,
	\qquad
	E_-\cap B=B\cap\operatorname{int}(H^c).
	\]
	In the constant case, one representative fills $B$ and
	the other is empty. 
	In particular, setting
	\(\Gamma:=\supp|DU|,\)
	we infer that for every $p\in\Gamma$ and every
	$0<\rho<h/2$, there is an affine hyperplane $P_p$
	through $p$ such that \(\Gamma\cap B_\rho(p)=P_p\cap B_\rho(p)\).
	Since balls of diameter less than $h$ cover $\mathbb R^d$, the local description shows that $E_+$ and $E_-$ are open and that $\partial E_+=\partial E_-=\Gamma$.
	
	\noindent
	\emph{Step 3: Complete parallel hyperplanes and their spacing.}
	The local supporting hyperplane is unique and locally
	constant along $\Gamma$. Indeed, if $q\in\Gamma\cap B_\rho(p)$, then $q\in P_p$, so $P_p\cap B_\rho(p)\cap B_\rho(q)$ is a relatively open neighborhood of $q$ in $P_p$. It is contained in $P_q$, hence $P_p=P_q$.
	Consequently, each connected component $C$ of $\Gamma$
	lies in a single affine hyperplane $P$.
	The local description makes $C$ relatively open in $P$.
	It is also relatively closed in $P$, because connected
	components of the closed set $\Gamma$ are closed.
	Since $P$ is connected, $C=P$.
	Thus $\Gamma$ is a union of complete affine hyperplanes.
	Distinct such hyperplanes are disjoint, since they belong to different connected components of $\Gamma$, and hence parallel.
	
	Their mutual distances are at least $h$.
	Indeed, if two planes had distance $a<h$, a ball centered
	halfway between them, with radius in $(a/2,h/2)$, would
	meet both. This contradicts Step~2, since that ball has
	diameter less than $h$ and can contain at most one
	planar interface. The family of planes is therefore
	locally finite.
	
	\noindent
	\emph{Step 4: One-dimensionality.}
	Since $U$ is nonconstant, $\Gamma$ is nonempty.
	Choose a common unit normal $\nu$ and write
	\[
	\Gamma=\bigcup_{t\in\mathcal S}
	\{x\in\R^d:x\cdot\nu=t\},
	\]
	where $\mathcal S\subset\R$ is closed and its distinct
	points have distance at least $h$.
	For every connected component $I$ of
	$\R\setminus\mathcal S$, the corresponding slab
	\[
	\{x\in\R^d:x\cdot\nu\in I\}
	\]
	is connected and has zero variation measure.
	Hence $U$ is constant almost everywhere on that slab.
	Define the function $g$ on $I$ to be this constant value.
	Then, by definition
	\[
	U(x)=g(x\cdot\nu)\qquad\text{$\mathcal{L}^d$-almost everywhere}.
	\]
	The local half-space description shows that the values
	switch across every interface, so the phases alternate.
	Since $\mathcal S$ is locally finite, $g$ has only finitely
	many jumps on each bounded interval, and therefore \(g\in BV_{\mathrm{loc}}(\R;\{-1,1\})\).
\end{proof}

The following result was used in the proof of Lemma \ref{lem:rigidity}. Although the result is probably standard, we could not find anywhere a rigorous proof, hence we sketch its proof.

\begin{lemma}[Euclidean monotone sets]
	\label{lem:euclidean-monotone}
	Let $B\subset\mathbb R^d$ be an open ball and let $A\subset B$
	be measurable. Suppose that, for almost every direction
	$e\in\mathbb S^{d-1}$ and for $\mathcal H^{d-1}$-almost every
	$y\in e^\perp$, both
	\[
	A\cap(y+\mathbb Re)
	\qquad\text{and}\qquad
	(B\setminus A)\cap(y+\mathbb Re)
	\]
	agree, up to $\mathcal H^1$-null sets, with intervals,
	possibly empty.
	Then there exists a half-space $H\subset\mathbb R^d$ such that
	\[
	|A\mathbin{\triangle}(B\cap H)|=0,
	\]
	allowing also the cases $|A|=0$ and $|B\setminus A|=0$.
\end{lemma}

\begin{proof}
	If $|A|=0$ or $|B\setminus A|=0$, the conclusion is immediate.
	
	We first prove that $A$ agrees almost everywhere with a convex set.
	The hypothesis concerning the sections of $A$, together with
	Fubini's theorem and the Blaschke--Petkantschin formula, implies
	\begin{equation}\label{eq:pair-interpolation}
		(1-t)x+ty\in A
		\quad\text{for almost every }(x,y,t)\in A\times A\times(0,1).
	\end{equation}
	Here and below, almost every refers to the appropriate product
	Lebesgue measure. To justify this implication, parametrize pairs
	of distinct points by their common line and their two coordinates
	on that line. On almost every such line $\ell$, the section
	$A\cap\ell$ agrees almost everywhere with an interval $I_\ell$.
	For almost every pair of endpoints in $A\cap\ell$, both endpoints
	belong to $I_\ell$ and are distinct. Their interpolating segment
	lies in $I_\ell$, and its points belong to $A$ for almost every
	interpolation parameter. The Blaschke--Petkantschin formula
	then transfers this assertion to almost every pair in $A\times A$ (see \cite[Section~7.2]{SW08}).
	
	We now iterate
	\eqref{eq:pair-interpolation}. For $k\geq 1$, consider
	\[
	\Delta_k^\circ
	=
	\left\{(\lambda_1,\ldots,\lambda_k):
	\lambda_i>0,\ \sum_{i=1}^k\lambda_i<1\right\},
	\qquad
	\lambda_0=1-\sum_{i=1}^k\lambda_i.
	\]
	Then
	\begin{equation}\label{eq:multiple-interpolation}
		\sum_{i=0}^k\lambda_i x_i\in A
		\quad\text{for almost every }
		(x_0,\ldots,x_k,\lambda)
		\in A^{k+1}\times\Delta_k^\circ.
	\end{equation}
	We prove this by induction on $k$. The case $k=1$ is
	\eqref{eq:pair-interpolation}. For the induction step, write
	\[
	\sum_{i=0}^k\lambda_i x_i=(1-t)z+t x_k,
	\qquad
	z=\sum_{i=0}^{k-1}\Lambda_i x_i,
	\]
	where $t=\lambda_k$ and $\Lambda_i=\lambda_i/(1-t)$.
	By the induction hypothesis, $z\in A$ for almost every choice
	of the preceding vertices and coefficients.
	
	For each fixed $\Lambda \in\Delta_{k-1}^\circ$, the map
	\[
	(x_0,\ldots,x_k,t)\longmapsto
	\left(\sum_{i=0}^{k-1}\Lambda_i x_i,\ x_k,\ t\right)
	\]
	is a surjective affine map. Its inverse image of a Lebesgue-null
	set is therefore null. Consequently, the exceptional set in
	\eqref{eq:pair-interpolation}, applied to $(z,x_k,t)$, remains
	null after this substitution. Fubini's theorem and the smooth
	change of coefficient variables
	$(\Lambda,t)\mapsto\lambda$ prove
	\eqref{eq:multiple-interpolation}.
	Taking $k=d$ and applying Fubini, we obtain
	\begin{equation}\label{eq:good-simplices}
		\bigl|\operatorname{co}\{x_0,\ldots,x_d\}\setminus A\bigr|=0
		\quad\text{for almost every }(x_0,\ldots,x_d)\in A^{d+1}.
	\end{equation}
	Indeed, for affinely independent vertices the barycentric map
	from $\Delta_d^\circ$ onto the interior of their simplex is an
	invertible affine map. For affinely dependent vertices, the
	simplex has zero $d$-dimensional measure.
	
	Let \(E:=A^{(1)}\cap B\), and \(K:=\operatorname{co}E\).
	By the Lebesgue density theorem, $|A\setminus K|=0$.
	We show that every simplex with vertices in $E$ is contained
	in $A$ up to a null set. Fix affinely independent
	$p_0,\ldots,p_d\in E$ and set $S:=\operatorname{co}\{p_0,\ldots,p_d\}$.
	Every sufficiently small product neighborhood of
	$(p_0,\ldots,p_d)$ meets $A^{d+1}$ in a set of positive measure.
	We can therefore choose tuples
	$(x_0^{(n)},\ldots,x_d^{(n)})$ satisfying
	\eqref{eq:good-simplices}, with $x_i^{(n)}\to p_i$ for every $i$.
	Set $S_n=\operatorname{co}\{x_0^{(n)},\ldots,x_d^{(n)}\}$.
	Every point of $\operatorname{int}S$ belongs to $S_n$ for all
	sufficiently large $n$. Since $|\partial S|=0$, Fatou's lemma
	gives
	\[
	|S\setminus A|
	\leq \liminf_{n\to\infty}|S_n\setminus A|=0.
	\]
	Choose a countable dense subset $D\subset E$. By
	Carath\'eodory's theorem, $\operatorname{co}D$ is a countable
	union of simplices with vertices in $D$. Hence $|(\operatorname{co}D)\setminus A|=0.$
	Moreover, we have $\overline{\operatorname{co}D}=\overline K.$
	Convex sets with the same closure have the same interior, and
	the boundary of a convex set is Lebesgue-null. Thus
	$|K\setminus\operatorname{co}D|=0$, and consequently
	$|K\setminus A|=0$. We thus have proved $|A\triangle K|=0$.
	Notice also that $K\subset B$, since $B$ is convex.
	
	Applying the same argument to $B\setminus A$ yields a convex
	set $K'\subset B$ such that $|(B\setminus A)\triangle K'|=0.$
	Both $K$ and $K'$ have nonempty interior, because both have
	positive measure. Their interiors are disjoint, since
	$|K\cap K'|=0$. If one of them has
	empty interior, it is null and we choose $H$ disjoint from $B$ or
	containing $B$. Otherwise, the separation theorem for convex sets \cite[Theorem~11.3]{Roc70} gives a half-space
	$H$ with $\operatorname{int}K\subset H$ and
	$\operatorname{int}K'\subset\mathbb R^d\setminus H$. Since
	$|B\setminus(K\cup K')|=0$ and convex sets have null boundary,
	we conclude $|A\triangle(B\cap H)|=0$.
\end{proof}

\subsection{Proof of the periodic theorem}
We are finally ready to prove Theorem \ref{thm:periodic}.
\begin{proof}[Proof of Theorem~\ref{thm:periodic}]
Multiplying \eqref{eq:comparison} by $\gamma$ and using $J=\gamma j_\lambda(h)$ proves \eqref{eq:mainper}.

Suppose now that equality holds in \eqref{eq:mainper}.
Then both nonnegative terms in \eqref{eq:exactdefect} vanish.
In particular,
\[
\big\langle\mu_{\lambda,h}^{(d)},
\psi_U(|\cdot|)\big\rangle=0.
\]
For $d\le5$, positivity of the Lebesgue density of
$\mu_{\lambda,h}^{(d)}$ in $B_h$, together with continuity
and non-negativity of $\psi_U$, implies
$\psi_U(r)=0$ for every $0<r<h$.
For $d=6,7$, since $\psi_U$ is
non-negative and non-decreasing, the pairing in \eqref{eq:liftedradialpairing} is non-negative, hence, the vanishing of the pairing implies $\psi_0=0$ on $[0,h]$ and $\psi_0'=0$ on $(0,h]$.

Thus, in every dimension under consideration,
\[
\fint_{\mathbb S^{d-1}}\defect_U(re)\,d\HH^{d-1}(e)
=\psi_U(r)=0
\qquad(0<r<h).
\]
Since $\defect_U$ is continuous and nonnegative,
its restriction to each sphere $\partial B_r$ must
vanish identically. Together with $\defect_U(0)=0$,
this yields
\[
\defect_U(z)=0
\qquad\text{for every }|z|<h.
\]
The function $U$ has zero mean. Indeed, this is imposed for $\lambda=0$, and follows from Corollary~\ref{cor:mean} for $\lambda>0$. It is therefore non-constant. Lemma~\ref{lem:rigidity} yields
$U(x)=g(\nu\cdot x)$ with parallel interfaces spaced at least $h$ apart for some $\nu \in \mathbb{S}^{d-1}$.

Tangential translation invariance implies that every nonzero Fourier frequency with $\widehat U(q)\ne0$ is parallel to $\nu$. To verify this, for each $v\perp\nu$ one has $U(\cdot+v)=U$, hence
$e^{iq\cdot v}\widehat U(q)=\widehat U(q)$. Varying $v$ gives $q\perp\nu^\perp$.
Vanishing of the spectral term in \eqref{eq:exactdefect} further implies that
if \(\widehat U(q)\ne0\) for \(q\ne0\), then \(h|q|\in(2\N-1)\pi\).

Since such $q$'s are parallel to $\nu$, we get that
$e^{iq\cdot(h\nu)}=-1$. The zero Fourier coefficient vanishes, so uniqueness of Fourier series gives
\begin{equation}\label{eq:antiperiod}
 U(x+h\nu)=-U(x)\quad\text{almost everywhere},
 \qquad g(t+h)=-g(t)\quad\text{almost everywhere}.
\end{equation}
Thus, \eqref{eq:antiperiod} exchanges the two phases and leaves their interface family invariant. That is, if $S$ is the set of jump points of $g$, then $S+h=S$.
Choose $t_0\in S$. The set $t_0+h\Z$ is contained in $S$. No additional jump point can lie between two consecutive points of this grid, because distinct points have distance at least $h$. Hence $S=t_0+h\Z$. The phases alternate, so $g$ is a square wave with layer width exactly $h$. Its period group is exactly $2h\Z$.
Using the fact that \(U\) is \(L\)-periodic in each coordinate, we infer that for every $i=1,\dots,d$,
\[
U(x+Le_i)
=g\bigl((x+Le_i)\cdot\nu\bigr)
=g(x\cdot\nu+L\nu_i).
\]
Therefore, \(g(t+L\nu_i)=g(t)\) for $\mathcal{L}^1$-almost every $t$.
So \(L\nu_i\) must be a period of \(g\), which in turn means \(L\nu_i=2h\,a_i\)
for some $a_i \in \Z$.
Since \(|\nu|=1\),
\[
|a|=\frac{L}{2h}|\nu|=\frac{L}{2h}>0,
\]
so \(a\ne0\) and the phase of the square wave gives exactly \eqref{eq:periodicminimizers}.

Conversely, let $U$ have the form \eqref{eq:periodicminimizers}. Notice that by symmetry it has zero mean and variation $|DU|(T)=2V/h$.
Its Fourier support consists of odd multiples of $2\pi a/L$. In particular, for every $q \in (2\pi/L)\Z^d$ with $q \neq 0$ such that $\widehat{U}(q)\neq0$, since $|a|=L/2h$, we have $h|q| \in (2\N-1)\pi$. Thus the spectral term in \eqref{eq:exactdefect} vanishes. If $|z|\le h$, then $|z\cdot\nu|\le h$, and direct translation of the square wave gives
\[
 \trn_U(z)=\frac{2V}{h}|z\cdot\nu|=\varn_U(z).
\]
This implies that the geometric term in \eqref{eq:exactdefect} also vanishes since the measure is supported in $\overline B_h$. This proves equality and completes the classification.
\end{proof}

\subsection{Even reflection and the Neumann functional}
Before proving Theorem \ref{thm:neumann} we need a preliminary result.
\begin{lemma}[Exact reflection identities]\label{lem:reflection}
Let $u\in\A_\lambda(Q_L)$. Define its even reflection on $(-L,L)^d$ by
\[
 U(x):=u(|x_1|,\ldots,|x_d|),
\]
and extend it with period $2L$ in every coordinate. Then
\begin{equation}\label{eq:reflectionBV}
 |DU|(\T_{2L}^d)=2^d|Du|(Q_L),\qquad
 \int_{\T_{2L}^d}U=2^d\int_{Q_L}u,
\end{equation}
and
\begin{equation}\label{eq:reflectionN}
 N_\lambda^{\mathrm{per}}(U)=2^dN_\lambda^N(u).
\end{equation}
\end{lemma}
\begin{proof}
We first observe that on each open reflected copy of $Q_L$, reflection is an isometry and preserves total variation. Across a coordinate face $x_i=0$, the two one-sided traces agree by evenness. Moreover, the same holds across the periodic identification $x_i=-L\sim L$. Therefore, there is no additional jump term on these faces (see e.g. \cite[Theorem 3.84]{AFP}). This entails $U \in BV(\T_{2L}^d;\{-1,1\})$ and \eqref{eq:reflectionBV}.

To prove \eqref{eq:reflectionN} let $v=R_\lambda^Nu$ on $Q_L$ (with mean-zero normalization if $\lambda=0$), and let $W$ be its even reflection. Since again the traces match along the faces, we get $W\in H^1(\T_{2L}^d)$. For a smooth periodic test function $\psi$ defined on $(-L,L)^d$, set
\[
 \varphi(y):=\sum_{\varepsilon\in\{-1,1\}^d}
 \psi(\varepsilon_1y_1,\ldots,\varepsilon_dy_d),
 \qquad y\in Q_L.
\]
Clearly $\varphi \in H^1(Q_L)$ and satisfies
\[
 \partial_i\varphi(y)=
 \sum_{\varepsilon\in\{-1,1\}^d} \varepsilon_i(\partial_i\psi)(\varepsilon y).
\]
We now use $\varphi$ in \eqref{eq:weakN}. Since
$\partial_iW(\varepsilon y)=\varepsilon_i\partial_i v(y)$,
\[
 \int_{\T_{2L}^d}(\nabla W\cdot\nabla\psi+\lambda^2 W\psi)\dd x
 =\int_{\T_{2L}^d}U\psi\dd x.
\]
By density, the identity holds for every $\psi\in H^1(\T_{2L}^d)$. For $\lambda>0$ this characterizes $R_\lambda^{\mathrm{per}}U$. For $\lambda=0$, $U$ and the reflected potential both have zero mean, and the zero-mean solution is unique. Hence the reflected potential equals $R_\lambda^{\mathrm{per}}U$. Therefore, using a change of variables,
\[
 N_\lambda^{\mathrm{per}}(U)
 =\int_{\T_{2L}^d}UW\dd x
 =2^d\int_{Q_L}uv\dd x
 =2^dN_\lambda^N(u).
\]
This concludes the proof.
\end{proof}

\subsection{Proof of the Neumann theorem}
We now prove Theorem \ref{thm:neumann}, it will follow as a consequence of Theorem \ref{thm:periodic}.
\begin{proof}[Proof of Theorem~\ref{thm:neumann}]
	Let $u\in\A_\lambda(Q_L)$, and like in Lemma \ref{lem:reflection}, define its even extension by \(U(x)=u(|x_1|,\ldots,|x_d|)\) for \(x\in(-L,L)^d\)
	followed by $2L$-periodic extension in each coordinate.
	By Lemma~\ref{lem:reflection},
	\[
	|DU|(\T_{2L}^d)=2^d|Du|(Q_L),
	\qquad
	N_\lambda^{\mathrm{per}}(U)=2^dN_\lambda^N(u).
	\]
	This gives that \(\E_{\lambda,J,\gamma}^{\mathrm{per}}(U)=2^d\E_{\lambda,J,\gamma}^N(u)\).
	Moreover, \(\mean U=\mean u\) by \eqref{eq:reflectionBV}.
	In particular, when $\lambda=0$, the mean-zero constraint
	is preserved, so $U$ is an admissible periodic competitor.
	Since $J=\gamma j_\lambda(h)$, Proposition~\ref{prop:comparison}
	on $\T_{2L}^d$ gives \(\E_{\lambda,J,\gamma}^{\mathrm{per}}(U)
	\ge\gamma\sigma_\lambda(h)(2L)^d\).
	Dividing by $2^d$ we obtain
	\[
	\E_{\lambda,J,\gamma}^N(u)
	\ge\gamma\sigma_\lambda(h)L^d,
	\]
	which is \eqref{eq:mainN}.
	For $\lambda>0$, applying
	Corollary~\ref{cor:mean} to $U$ gives
	\[
	\E_{\lambda,J,\gamma}^N(u)
	-\gamma\sigma_\lambda(h)L^d
	\ge
	\frac{\gamma L^d}{\lambda^2}
	\sech^2\left(\frac{\lambda h}2\right)\,|\overline u|^2.
	\]
	Thus equality forces $\mean u=0$ also in the
	screened case.

	Suppose now that equality holds in \eqref{eq:mainN}, that is
	\[
	\E_{\lambda,J,\gamma}^{\mathrm{per}}(U)
	=\gamma\sigma_\lambda(h)(2L)^d.
	\]
	By Theorem~\ref{thm:periodic} applied with domain size $2L$, we infer that
	there exist $a\in\Z^d\setminus\{0\}$ and $\varphi\in\R$
	such that
	\[
	U(x)=\sgn\cos(q\cdot x+\varphi),
	\qquad
	q=\frac{\pi}{L}a,
	\qquad
	|a|=\frac{L}{h}.
	\]
	In particular, \(q\in \pi\Z^d/L\) and \(|q|=\pi/h\).

	For each $i\in\{1,\ldots,d\}$, let us now define \(R_i(x_1,\ldots,x_d)
	:=(x_1,\ldots,-x_i,\ldots,x_d)\).
	By construction of the even extension, we have that \(U(R_i x)=U(x)\) $\mathcal{L}^d$-almost everywhere.
	
	The interfaces of the periodic lift of $U$,
	$\Gamma:=\supp|DU|$ in $\R^d$, are given by
	\[
	\Gamma
	=
	\bigcup_{k\in\Z}
	\left\{
	x\in\R^d:
	q\cdot x+\varphi=\frac{\pi}{2}+k\pi
	\right\}.
	\]
	In particular, $\Gamma$ is a family of complete parallel hyperplanes with
	normal $q$. The identity $U\circ R_i=U$ implies
	$R_i\Gamma=\Gamma$. On the other hand, reflecting these
	hyperplanes gives a family with normal $R_iq$.
	Since the interface family has a unique normal line,
	$R_iq$ must be parallel to $q$. This implies \(R_iq=\pm q\)
	for every $i=1,\dots,d$.
	
	Suppose now that $q_i\ne0$ and $q_k\ne0$ for distinct indices
	$i$ and $k$. Reflection in the $i$-th coordinate changes
	$q_i$ to $-q_i$, so $R_iq\ne q$. It leaves $q_k$ unchanged,
	so $R_iq\ne -q$. This is a contradiction.
	Therefore $q$ has exactly one nonzero coordinate.
	
	Since $q\in(\pi/L)\Z^d$, there exist
	$j\in\{1,\ldots,d\}$ and $n\in\N$ with $n \geq 1$ such that \(q=\pm\frac{\pi n}{L}e_j\). 
	Thus, $nh=L$ by the condition on $|q|$.
	In particular, $L/h\in\N$.
	
	Replacing $\varphi$ by $-\varphi$ if necessary, we may
	choose the positive orientation and write
	\[
	U(x)=\sgn\cos\left(\frac{\pi n x_j}{L}+\varphi\right).
	\]
	Set $\kappa:=\pi n/L$. We now show that $\varphi \in \pi \Z$.
	The fundamental Fourier terms of the square wave are
	\[
	\sgn\cos(\kappa x_j+\varphi)
	=
	\frac{2}{\pi}e^{i\varphi}e^{i\kappa x_j}
	+\frac{2}{\pi}e^{-i\varphi}e^{-i\kappa x_j}
	+\text{(harmonics of order $\geq 3$)},
	\]
	with the expansion understood in $L^2$.
	Thus, with the normalized Fourier coefficients on the torus,
	\[
	\widehat U(\kappa e_j)=\frac{2}{\pi}e^{i\varphi},
	\qquad
	\widehat U(-\kappa e_j)=\frac{2}{\pi}e^{-i\varphi}.
	\]
	
	Reflection in the $j$-th coordinate interchanges these
	two frequencies. Since $U\circ R_j=U$, their coefficients
	must agree, giving \(e^{i\varphi}=e^{-i\varphi}.\)
	It follows that $\sin\varphi=0$, or equivalently,
	$\varphi\in\pi\Z$. Therefore,
	\[
	U(x)=\pm\sgn\cos\left(\frac{\pi n x_j}{L}\right).
	\]
	Restricting $U$ to $Q_L$ gives exactly
	\eqref{eq:neumannminimizers}.

	Conversely, suppose now $n=L/h\in\N$, and let
	\[
	u(x)=\varepsilon\sgn\cos\left(\frac{\pi n x_j}{L}\right),
	\qquad \varepsilon\in\{-1,1\}.
	\]
	This function belongs to $BV(Q_L;\{-1,1\})$ and it has zero mean.
	Thus, $u$ is admissible for every
	$\lambda\ge0$, including $\lambda=0$.
	
	Its even periodic extension on $\T_{2L}^d$ is given by
	\[
	U(x)=\varepsilon\sgn\cos\left(\frac{\pi n x_j}{L}\right).
	\]
	Indeed, this expression is even in every coordinate
	and $2L$-periodic. Its fundamental frequency is
	$q=(\pi n/L)e_j$, with $|q|=\pi/h$. It therefore belongs to the periodic
	equality class in Theorem~\ref{thm:periodic}.
	This implies that
	\[
	\E_{\lambda,J,\gamma}^{\mathrm{per}}(U)
	=\gamma\sigma_\lambda(h)(2L)^d.
	\]
	Applying the reflection identities of Lemma \ref{lem:reflection} once more yields
	\[
	\E_{\lambda,J,\gamma}^N(u)
	=2^{-d}\E_{\lambda,J,\gamma}^{\mathrm{per}}(U)
	=\gamma\sigma_\lambda(h)L^d.
	\]
	These profiles attain the lower bound and are therefore
	global minimizers. 
\end{proof}

\appendix
\section{Limits of positivity in higher dimensions}
\label{app:dimension}

The comparison distributions constructed in
Proposition~\ref{prop:measure} are finite positive measures
in dimensions $d\le5$. In dimensions six and seven,
their action remains non-negative on non-negative,
non-decreasing radial test functions, which is sufficient
for the sharp comparison inequality.
Nevertheless, positivity as a measure cannot hold in
any dimension $d\ge6$.

More precisely, consider the radial multiplier
\begin{equation}\label{eq:appF}
	\Fp(\xi)=
	\frac{\cosh(\lambda h)-\cos(h|\xi|)}
	{(1+\cosh(\lambda h))(\lambda^2+|\xi|^2)},
\end{equation}
with the value at $\lambda=0$ and $\xi=0$ understood
by continuity.

\begin{proposition}\label{prop:nogo}
	For every $h>0$, $\lambda\ge0$, and $d\ge6$, there is
	no finite positive Borel measure on $\R^d$ whose Fourier
	transform equals \eqref{eq:appF} at every real frequency.
\end{proposition}

Before proving the Proposition, we need a preliminary lemma.

\begin{lemma}\label{lem:radialprojection6}
	Let $\eta$ be a finite positive rotation-invariant measure on $\R^6$. Away from the origin, its projection onto $\R^4$ has a radial density $g \colon (0,\infty) \to [0,+\infty)$ satisfying
	\begin{equation}\label{eq:projectionbound}
		g(r)\le\frac{2\eta(\R^6)}{\pi^2r^4}\qquad \mbox{for $r>0$}.
	\end{equation}
	In particular, $g$ is bounded on every annulus bounded away from the origin. It also has a nonincreasing radial representative.
\end{lemma}
\begin{proof}
	By rotation invariance, \(\eta\) decomposes into normalized surface measures on spheres, weighted by the distribution of its mass with respect to the radius, together with a possible atom at the origin. Let $\nu$ denote this distribution on $(0,\infty)$.
	
	The radius-$s$ sphere in $\R^6$ is parameterized, away from a surface-null subset, by
	\[
	(y,\vartheta)\longmapsto
	\left(y,\sqrt{s^2-|y|^2}\cos\vartheta,
	\sqrt{s^2-|y|^2}\sin\vartheta\right),
	\quad y\in B_s^4,\quad 0\le\vartheta<2\pi.
	\]
	Its surface element is given by $s\,dy\dd\vartheta$. Since
	$\HH^5(\partial B_s^6)=\pi^3s^5$, the projection of normalized surface measure has density
	\[
	\frac{2\pi s}{\pi^3s^5}\one_{\{|y|<s\}}
	=\frac{2}{\pi^2s^4}\one_{\{|y|<s\}}.
	\]
	Hence, for $r>0$
	\begin{equation}
		g(r)=\int_{(r,\infty)}\frac{2}{\pi^2s^4}\dd\nu(s) \leq \frac{2\nu((0,\infty))}{\pi^2r^4}\leq\frac{2\eta(\R^6)}{\pi^2r^4}
	\end{equation}
	The quantity on the left hand side above is non-increasing in $r$ by definition. This proves \eqref{eq:projectionbound}.
\end{proof}

\begin{proof}[Proof of Proposition~\ref{prop:nogo}]
	Suppose first that a finite positive measure $\eta$ on $\R^6$ has Fourier transform $\Fp$. Since $\Fp$ is radial, uniqueness of Fourier transforms of finite measures implies that $\eta$ is rotation invariant.
	
	We now project $\eta$ onto $\R^4$. Its Fourier transform is the restriction of $\Fp$ to $\R^4$. By Proposition~\ref{prop:measure}, this is also the Fourier transform of $\mu_{\lambda,h}^{(4)}$, so the two projected measures must coincide.
	The four-dimensional density in \eqref{eq:k4} is
	\begin{equation}\label{eq:appk4}
		k_{\lambda,h}^{(4)}(r)
		=2\int_0^{\sqrt{h^2-r^2}}
		b_{\lambda,h}\!\left(\sqrt{r^2+t^2}\right)\dd t
		+\frac{2\Alh}{h\sqrt{h^2-r^2}},
		\qquad 0<r<h.
	\end{equation}
	Its integral term is nonnegative. Therefore, recalling $\Alh=1/(8\pi^2(1+\Cosh))$, with $\Cosh=\cosh(\lambda h)$, we infer
	\[
	k_{\lambda,h}^{(4)}(r)\ge\frac{2\Alh}{h\sqrt{h^2-r^2}}
	\longrightarrow+\infty \qquad \mbox{as $r \to h$}.
	\]
	This contradicts Lemma~\ref{lem:radialprojection6}, which bounds the projected density uniformly on $h/2<r<h$. 
	
	This proves nonexistence in dimension six. If a positive finite measure in dimension $d>6$ had transform $\Fp$, its projection onto a six-dimensional subspace would provide the excluded positive measure there. The conclusion thus follows for every $d\ge6$.
\end{proof}

\begin{remark}[Further dimensional lifts]
	Proposition~\ref{prop:nogo} concerns positivity as a
	measure and does not preclude the use of distributional
	comparison kernels. Indeed, for $d=4,5$, the identity
	\[
	\big\langle\mu_{\lambda,h}^{(d+2)},\psi_0(|\cdot|)
	\big\rangle
	=
	\int_{\R^d}
	\left[
	\psi_0(|z|)+\frac{|z|}{d}\psi_0'(|z|)
	\right]\,d\mu_{\lambda,h}^{(d)}(z)
	\]
	shows that non-negativity and monotonicity of the radial
	profile suffice to obtain a non-negative pairing in
	dimensions six and seven.
	
	A further lift introduces a second derivative.
	For smooth radial tests, applying the same identity twice
	would give
	\[
		\big\langle\mu_{\lambda,h}^{(d+4)},
		\psi_0(|\cdot|)\big\rangle=
		\int_{\R^d} \left[\psi_0(|z|)+\frac{2d+3}{d(d+2)}|z|\psi_0'(|z|)
		+\frac{|z|^2}{d(d+2)}\psi_0''(|z|)
		\right]\,d\mu_{\lambda,h}^{(d)}(z),
		\qquad d=4,5.
	\]
	Here $\mu_{\lambda,h}^{(d+4)}$ denotes the inverse
	Fourier transform of \eqref{eq:appF} in dimension $d+4$,
	defined by continuing the distributional Fourier walk.
	Thus dimensions eight and nine involve second
	derivatives of the radial test, and subsequent lifts
	introduce higher derivatives.
	
	Non-negativity and monotonicity of $\psi_0$ alone do
	not guarantee a non-negative pairing in this formula.
	The properties of the geometric defect established
	above therefore no longer yield positivity directly.
	Whether its additional geometric structure can be
	used to extend the sharp comparison to dimensions
	$d\ge8$ remains open.
\end{remark}


\begin{thebibliography}{99}
	
	\bibitem{ACC22}
	E.~Acerbi, C.-N.~Chen, and Y.~S.~Choi,
	\emph{Stability of lamellar configurations in a nonlocal
		sharp interface model},
	SIAM J. Math. Anal. \textbf{54} (2022), no.~1, 558--594.
	
	\bibitem{AFM13}
	E.~Acerbi, N.~Fusco, and M.~Morini,
	\emph{Minimality via second variation for a nonlocal
		isoperimetric problem},
	Comm. Math. Phys. \textbf{322} (2013), no.~2, 515--557.
	
	\bibitem{ABCT20}
	S.~Alama, L.~Bronsard, R.~Choksi, and I.~Topaloglu,
	\emph{Droplet phase in a nonlocal isoperimetric problem
		under confinement},
	Commun. Pure Appl. Anal. \textbf{19} (2020), no.~1, 175--202.
	
	\bibitem{ACO09}
	G.~Alberti, R.~Choksi, and F.~Otto,
	\emph{Uniform energy distribution for an isoperimetric problem
		with long-range interactions},
	J. Amer. Math. Soc. \textbf{22} (2009), no.~2, 569--605.
	
	\bibitem{AFP}
	L.~Ambrosio, N.~Fusco and D.~Pallara,
	\emph{Functions of Bounded Variation and Free Discontinuity Problems},
	Oxford Mathematical Monographs, Oxford University Press, 2000.
	
	\bibitem{BC14}
	M.~Bonacini and R.~Cristoferi,
	\emph{Local and global minimality results for a nonlocal
		isoperimetric problem on $\mathbb{R}^N$},
	SIAM J. Math. Anal. \textbf{46} (2014), no.~4, 2310--2349.
	
	\bibitem{CCA15}
	B.~Chacko, C.~Chalmers, and A.~J.~Archer,
	\emph{Two-dimensional colloidal fluids exhibiting pattern formation},
	J. Chem. Phys. \textbf{143} (2015), 244904.
	
	\bibitem{CG26}
	O.~Chodosh and M.~Gianocca,
	\emph{No compromise in the liquid drop model},
	preprint (2026), \texttt{arXiv:2608.11517}.
	
	\bibitem{CR25}
	O.~Chodosh and I.~Ruohoniemi,
	\newblock \emph{On minimizers in the liquid drop model},
	\newblock Comm. Pure Appl. Math. \textbf{78} (2025), 366--381.
	
	\bibitem{CP10}
	R.~Choksi and M.~A.~Peletier,
	\emph{Small volume fraction limit of the diblock copolymer
		problem: I. Sharp-interface functional},
	SIAM J. Math. Anal. \textbf{42} (2010), no.~3, 1334--1370.
	
	\bibitem{CP11}
	R.~Choksi and M.~A.~Peletier,
	\emph{Small volume fraction limit of the diblock copolymer
		problem: II. Diffuse-interface functional},
	SIAM J. Math. Anal. \textbf{43} (2011), no.~2, 739--763.
	
	\bibitem{CR03}
	R.~Choksi and X.~Ren,
	\emph{On the derivation of a density functional theory for
		microphase separation of diblock copolymers},
	J. Stat. Phys. \textbf{113} (2003), nos.~1--2, 151--176.
	
	\bibitem{CS07}
	R.~Choksi and P.~Sternberg,
	\emph{On the first and second variations of a nonlocal
		isoperimetric problem},
	J. Reine Angew. Math. \textbf{611} (2007), 75--108.
	
	\bibitem{CiSp13}
	M.~Cicalese and E.~Spadaro,
	\emph{Droplet minimizers of an isoperimetric problem
		with long-range interactions},
	Comm. Pure Appl. Math. \textbf{66} (2013), no.~8, 1298--1333.
	
	\bibitem{CE03}
	H.~Cohn and N.~Elkies,
	\emph{New upper bounds on sphere packings I},
	Ann. of Math. (2) \textbf{157} (2003), no.~2, 689--714.
	
	
	\bibitem{CKMRV17}
	H.~Cohn, A.~Kumar, S.~D.~Miller, D.~Radchenko,
	and M.~Viazovska,
	\emph{The sphere packing problem in dimension $24$},
	Ann. of Math. (2) \textbf{185} (2017), no.~3, 1017--1033.
	
	
	\bibitem{Cri18}
	R.~Cristoferi,
	\emph{On periodic critical points and local minimizers
		of the Ohta--Kawasaki functional},
	Nonlinear Anal. \textbf{168} (2018), 81--109.
	
	\bibitem{DKR22}
	S.~Daneri, A.~Kerschbaum, and E.~Runa,
	\emph{One-dimensionality of the minimizers for a diffuse
		interface generalized antiferromagnetic model in general
		dimension},
	J. Funct. Anal. \textbf{283} (2022), no.~12, 109715.
	
	\bibitem{DR19}
	S.~Daneri and E.~Runa,
	\emph{Exact periodic stripes for minimizers of a local/nonlocal
		interaction functional in general dimension},
	Arch. Ration. Mech. Anal. \textbf{231} (2019), no.~1, 519--589.
	
	\bibitem{DR20}
	S.~Daneri and E.~Runa,
	\emph{Pattern formation for a local/nonlocal interaction
		functional arising in colloidal systems},
	SIAM J. Math. Anal. \textbf{52} (2020), no.~3, 2531--2560.
	
	\bibitem{DR24}
	S.~Daneri and E.~Runa,
	\emph{A rigorous approach to pattern formation for isotropic
		isoperimetric problems with competing nonlocal interactions},
	preprint, arXiv:2406.13773 (2024; revised 2025).
	
	\bibitem{FFMMM15}
	A.~Figalli, N.~Fusco, F.~Maggi, V.~Millot, and M.~Morini,
	\emph{Isoperimetry and stability properties of balls
		with respect to nonlocal energies},
	Comm. Math. Phys. \textbf{336} (2015), no.~1, 441--507.
	
	\bibitem{Fio26}
	A.~Fiorini,
	\emph{Local and global minimality of the lamella for the
		anisotropic Ohta--Kawasaki energy},
	preprint, arXiv:2604.13736 (2026).
	
	
	
	\bibitem{FL19}
	R.~L.~Frank and E.~H.~Lieb,
	\newblock \emph{Periodic energy minimizers for a one-dimensional liquid drop model},
	\newblock Lett. Math. Phys. \textbf{109} (2019), no.~9, 2069--2081.
	
	\bibitem{FLS26}
	R.~L.~Frank, M.~Lewin, and R.~Seiringer,
	\newblock \emph{Liquid drop model for nuclear matter in the low density limit},
	\newblock Comm. Pure Appl. Math. (2026), e70039.
	
	\bibitem{Gal11}
	B.~Galerne,
	\emph{Computation of the perimeter of measurable sets via their
		covariogram. Applications to random sets},
	Image Anal. Stereol. \textbf{30} (2011), no.~1, 39--51.
	
	\bibitem{GLL06}
	A.~Giuliani, J.~L.~Lebowitz, and E.~H.~Lieb,
	\emph{Ising models with long-range antiferromagnetic and
		short-range ferromagnetic interactions},
	Phys. Rev. B \textbf{74} (2006), 064420.
	
	\bibitem{GLL09}
	A.~Giuliani, J.~L.~Lebowitz, and E.~H.~Lieb,
	\emph{Periodic minimizers in 1D local mean field theory},
	Comm. Math. Phys. \textbf{286} (2009), no.~1, 163--177.
	
	\bibitem{GLS14}
	A.~Giuliani, E.~H.~Lieb, and R.~Seiringer,
	\emph{Formation of stripes and slabs near the ferromagnetic
		transition},
	Comm. Math. Phys. \textbf{331} (2014), no.~1, 333--350.
	
	\bibitem{GS16}
	A.~Giuliani and R.~Seiringer,
	\emph{Periodic striped ground states in Ising models
		with competing interactions},
	Comm. Math. Phys. \textbf{347} (2016), no.~3, 983--1007.
	
	\bibitem{GMS13}
	D.~Goldman, C.~B.~Muratov, and S.~Serfaty,
	\emph{The $\Gamma$-limit of the two-dimensional
		Ohta--Kawasaki energy. I. Droplet density},
	Arch. Ration. Mech. Anal. \textbf{210} (2013), no.~2, 581--613.
	
	\bibitem{GMS14}
	D.~Goldman, C.~B.~Muratov, and S.~Serfaty,
	\emph{The $\Gamma$-limit of the two-dimensional
		Ohta--Kawasaki energy. Droplet arrangement via the
		renormalized energy},
	Arch. Ration. Mech. Anal. \textbf{212} (2014), no.~2, 445--501.
	
	\bibitem{GR19}
	M.~Goldman and E.~Runa,
	\emph{On the optimality of stripes in a variational model
		with non-local interactions},
	Calc. Var. Partial Differential Equations
	\textbf{58} (2019), no.~3, Paper No.~103.
	
	\bibitem{GT13}
	L. Grafakos and G. Teschl,
	\emph{On Fourier transforms of radial functions and distributions},
	J. Fourier Anal. Appl. \textbf{19} (2013), 167--179.
	
	\bibitem{Jul14}
	V.~Julin,
	\emph{Isoperimetric problem with a Coulomb repulsive term},
	Indiana Univ. Math. J. \textbf{63} (2014), no.~1, 77--89.
	
	\bibitem{JP17}
	V.~Julin and G.~Pisante,
	\emph{Minimality via second variation for microphase separation
		of diblock copolymer melts},
	J. Reine Angew. Math. \textbf{729} (2017), 81--117.
	
	\bibitem{Ker22}
	A.~Kerschbaum,
	\newblock \emph{Striped patterns for generalized antiferromagnetic functionals
	with power law kernels of exponent smaller than $d+2$},
	\newblock Nonlinear Anal. \textbf{216} (2022), 112724.
	
	
	\bibitem{KM13}
	H.~Kn\"upfer and C.~B.~Muratov,
	\emph{On an isoperimetric problem with a competing nonlocal
		term I: The planar case},
	Comm. Pure Appl. Math. \textbf{66} (2013), no.~7, 1129--1162.
	
	\bibitem{KM14}
	H.~Kn\"upfer and C.~B.~Muratov,
	\emph{On an isoperimetric problem with a competing nonlocal
		term II: The general case},
	Comm. Pure Appl. Math. \textbf{67} (2014), no.~12, 1974--1994.
	
	\bibitem{KMN16}
	H.~Kn\"upfer, C.~B.~Muratov, and M.~Novaga,
	\emph{Low density phases in a uniformly charged liquid},
	Comm. Math. Phys. \textbf{345} (2016), no.~1, 141--183.
	
	\bibitem{KMN20}
	H.~Kn\"upfer, C.~B.~Muratov, and M.~Novaga,
	\emph{Emergence of nontrivial minimizers for the
		three-dimensional Ohta--Kawasaki energy},
	Pure Appl. Anal. \textbf{2} (2020), no.~1, 1--21.
	
	\bibitem{Mil26}
	E.~Milman,
	\emph{Spectral and isoperimetric bounds on flat tori},
	preprint, arXiv:2608.13052 (2026).
	
	\bibitem{MS14}
	M.~Morini and P.~Sternberg,
	\emph{Cascade of minimizers for a nonlocal isoperimetric
		problem in thin domains},
	SIAM J. Math. Anal. \textbf{46} (2014), no.~3, 2033--2051.
	
	\bibitem{Mu93}
	S.~M\"uller,
	\emph{Singular perturbations as a selection criterion for
		periodic minimizing sequences},
	Calc. Var. Partial Differential Equations
	\textbf{1} (1993), no.~2, 169--204.
	
	\bibitem{M10}
	C.~B.~Muratov,
	\emph{Droplet phases in non-local Ginzburg--Landau models
		with Coulomb repulsion in two dimensions},
	Comm. Math. Phys. \textbf{299} (2010), no.~1, 45--87.
	
	\bibitem{OK86}
	T.~Ohta and K.~Kawasaki,
	\emph{Equilibrium morphology of block copolymer melts},
	Macromolecules \textbf{19} (1986), no.~10, 2621--2632.
	
	\bibitem{RW00}
	X.~Ren and J.~Wei,
	\emph{On the multiplicity of solutions of two nonlocal
		variational problems},
	SIAM J. Math. Anal. \textbf{31} (2000), no.~4, 909--924.
	
	\bibitem{RW03}
	X.~Ren and J.~Wei,
	\emph{On energy minimizers of the diblock copolymer problem},
	Interfaces Free Bound. \textbf{5} (2003), no.~2, 193--238.
	
	\bibitem{RW05}
	X.~Ren and J.~Wei,
	\emph{Wriggled lamellar solutions and their stability in the
		diblock copolymer problem},
	SIAM J. Math. Anal. \textbf{37} (2005), no.~2, 455--489.
	
	\bibitem{Roc70}
	R.~T.~Rockafellar,
	\newblock \emph{Convex Analysis},
	\newblock Princeton Mathematical Series, Vol.~28,
	Princeton University Press, Princeton, NJ, 1970.
	
	\bibitem{SW08}
	R.~Schneider and W.~Weil,
	\newblock \emph{Stochastic and Integral Geometry},
	\newblock Probability and Its Applications, Springer, Berlin, 2008.
	
	\bibitem{Spa09}
	E.~N.~Spadaro,
	\emph{Uniform energy and density distribution: diblock copolymers' functional},
	Interfaces Free Bound. \textbf{11} (2009), no.~3, 447--474.
	
	\bibitem{ST11}
	P.~Sternberg and I.~Topaloglu,
	\emph{On the global minimizers of a nonlocal isoperimetric
		problem in two dimensions},
	Interfaces Free Bound. \textbf{13} (2011), no.~1, 155--169.
	
	\bibitem{Stradner04}
	A.~Stradner, H.~Sedgwick, F.~Cardinaux, W.~C.~K.~Poon,
	S.~U.~Egelhaaf, and P.~Schurtenberger,
	\emph{Equilibrium cluster formation in concentrated protein
		solutions and colloids},
	Nature \textbf{432} (2004), 492--495.
	
	\bibitem{V17}
	M.~S.~Viazovska,
	\emph{The sphere packing problem in dimension $8$},
	Ann. of Math. (2) \textbf{185} (2017), no.~3, 991--1015.
	
	
\end{thebibliography}
\end{document}